\documentclass[12pt]{amsart}
\usepackage[letterpaper]{geometry}
\usepackage{graphicx}
\usepackage[rgb]{xcolor}
\usepackage[colorlinks, allcolors=blue]{hyperref}
\usepackage{amsmath,amscd,amsthm,amssymb,tikz,graphicx,ifthen}
\usepackage{enumerate,multicol,subcaption}
\usetikzlibrary{calc,positioning,matrix,arrows,decorations.pathmorphing}

\newtheorem{thm}{Theorem}[section]
\newtheorem{cor}[thm]{Corollary}
\newtheorem{lem}[thm]{Lemma}
\newtheorem{prop}[thm]{Proposition}

\theoremstyle{definition}
\newtheorem{defn}[thm]{Definition}

\theoremstyle{remark}
\newtheorem{remark}[thm]{Remark}
\numberwithin{equation}{section}

\newcommand{\abs}[1]{\left\vert#1\right\vert}

\newcommand{\R}{{\mathbb R}}
\newcommand{\Z}{{\mathbb Z}}
\newcommand{\C}{{\mathbb C}}
\newcommand{\D}{{\mathbb D}}
\newcommand{\A}{{\bar A}}
\renewcommand{\P}{{\bar P}}
\newcommand{\Q}{{\bar Q}}
\newcommand{\Ab}{{\mathbb A}}
\newcommand{\Sb}{{\mathbb S}}
\newcommand{\F}{{\mathcal F}}
\newcommand{\G}{\Gamma}
\newcommand{\g}{\gamma}
\newcommand{\Bs}{{\mathcal B}}
\newcommand{\Cs}{{\mathcal C}}
\newcommand{\Gs}{{\mathcal G}}
\newcommand{\Os}{{\mathcal O}}
\newcommand{\Id}{\mathrm{Id}}
\DeclareMathOperator{\Cod}{Cod}
\newcommand\sequence[2][\A]{[#2]}

\newcommand\commutativeDiagram[9][1.25]{\raisebox{-3em}{\begin{tikzpicture}[scale=1.5]
    \node (A) at (0,1) {$#2$};
    \node (B) at (#1,1) {$#4$};
    \node (C) at (0,0) {$#7$};
    \node (D) at (#1,0) {$#9$};
	\path[->,font=\scriptsize,>=angle 90]
        (A) edge node [above] {$#3$} (B)
        (A) edge node [left]  {$#5$} (C)
        (B) edge node [right] {$#6$} (D)
        (C) edge node [below] {$#8$} (D);
\end{tikzpicture}}}

\title[Adler and Flatto revisited]{Adler and Flatto revisited: cross-sections for geodesic flow on compact surfaces of constant negative curvature}

\dedicatory{Dedicated to the memory of Roy Adler}

\author{Adam Abrams}
\address{Department of Mathematics, The Pennsylvania State University, University Park, PA 16802} 
\email{ajz5041@psu.edu}

\author{Svetlana Katok}
\address{Department of Mathematics, The Pennsylvania State University, University Park, PA 16802} 
\email{sxk37@psu.edu}

\date{}

\begin{document}

\begin{abstract}
	We describe a family of arithmetic cross-sections for geodesic flow on compact surfaces of constant negative curvature based on the study of generalized Bowen-Series boundary maps associated to cocompact torsion-free Fuchsian groups and their natural extensions, introduced in \cite{KU17}. If the boundary map satisfies the short cycle property, i.e., the forward orbits at each discontinuity point coincide after one step, the natural extension map has a global attractor with finite rectangular structure, and the associated arithmetic cross-section is parametrized by the attractor. This construction allows us to represent the geodesic flow as a special flow over a symbolic system of coding sequences. In special cases where the ``cycle ends'' are discontinuity points of the boundary maps, the resulting symbolic system is sofic. Thus we extend and in some ways simplify several results of Adler-Flatto's 1991 paper \cite{AF91}. We also compute the measure-theoretic entropy of the boundary maps.
\end{abstract}

\maketitle

\section{Introduction}\label{sec introduction}

Adler and Flatto have written three papers~\cite{AF82, AF84, AF91} devoted to representation of the geodesic flow on surfaces of constant negative curvature as special flow over a topological Markov chain. The first two papers are devoted to the modular surface and the third to the compact surface $M=\G\backslash\D$ of genus $g\ge 2$, where $\D=\{\, z \in \C : \abs z < 1 \,\}$ is the unit disk endowed with hyperbolic metric
\begin{equation}
    \label{hypmetric} \frac{2 \abs{dz} }{1-{\abs z}^2}
\end{equation}
and $\G$ is a finitely generated Fuchsian group of the first kind acting freely on $\D$.

Recall that geodesics in this model are half-circles or diameters orthogonal to $\Sb=\partial\D$, the circle at infinity. The geodesic flow $\{\tilde\varphi^t\}$ on $\D$ is defined as an $\R$-action on the unit tangent bundle $S\D$ that moves a tangent vector along the geodesic defined by this vector with unit speed. The geodesic flow $\{\tilde\varphi^t\}$ on $\D$ descends to the geodesic flow $\{\varphi^t\}$ on the factor $M=\Gamma\backslash\D$ via the canonical projection
\begin{equation}
    \label{projection} \pi: S\D \to SM
\end{equation}
of the unit tangent bundles.  The orbits of the geodesic flow $\{\varphi^t\}$ are oriented geodesics on $M$.

A classical (Ford) fundamental domain for $\G$ is a $4g$-sided regular polygon centered at the origin. In~\cite{AF91}, Adler and Flatto used another fundamental domain---an $(8g-4)$-sided polygon $\F$---that was much more convenient for their purposes. Its sides are geodesic segments which satisfy the {\em extension condition}: the geodesic extensions of these segments never intersect the interior of the tiling sets $\gamma\mathcal F$, $\gamma\in \G$.

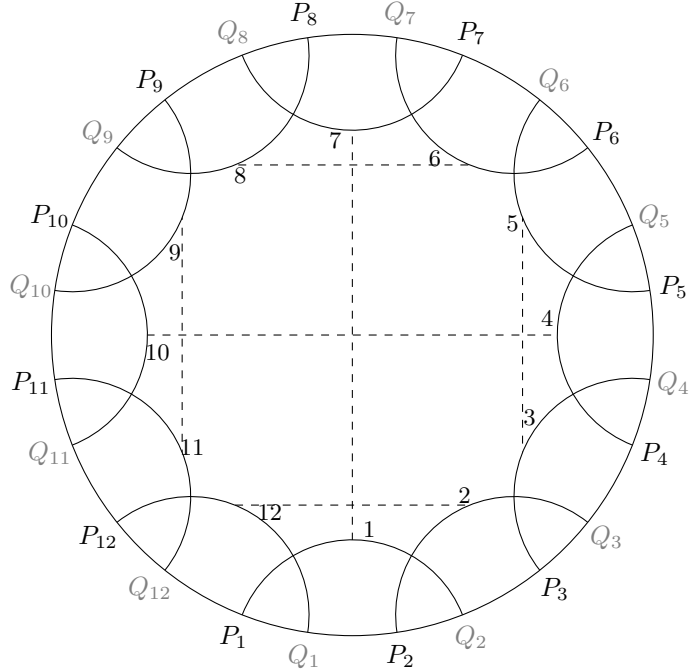
\begin{figure}
\begin{tikzpicture}[scale=4]
    % isometric circles 
    \foreach \k in {1,2,...,12} {
        \draw [rotate=\k*30] (0:1.07457)+(111.471:0.39332) arc (111.471:248.529:0.39332);
        \draw (\k*30-120+5:0.65) node {\scriptsize$\k$};
        
        \draw (-141.471+\k*30:1.08) node [black] {\footnotesize$P_{\k}$};
        \draw (-128.529+\k*30:1.08) node [gray] {\footnotesize$Q_{\k}$};
    }
    \draw (0,0) circle (1);
    
    % side-pairing 
  	\clip (0.733945,-0.19666) arc (-150:-210:0.39332) arc (-120:-180:0.39332) arc (-90:-150:0.39332) arc (-60:-120:0.39332) arc (-30:-90:0.39332) arc (0:-60:0.39332) arc (30:-30:0.39332) arc (60:0:0.39332) arc (90:30:0.39332) arc (120:60:0.39332) arc (150:90:0.39332) arc (180:120:0.39332);
	\draw [dashed] (0,-1) -- (0,1);
	\draw [dashed] (-1,0) -- (1,0);
	\draw [dashed] (225:0.8) rectangle (45:0.8);
    
\end{tikzpicture}\vspace*{-1em}
\caption{Side-pairing of the fundamental domain $\F$ for genus $g=2$.}
\label{fig disk}
\end{figure}

The fundamental polygon $\F$ will play an important role in this paper, so let us describe it in detail. We label the sides of $\F$ in a counterclockwise order by numbers $1 \le i \le 8g-4$ and label the vertices of $\F$ by $V_i$ so that side $i$ connects $V_i$ to $V_{i+1} \pmod {8g-4}$. 

We denote by $P_i$ and $Q_{i+1}$ the endpoints of the oriented infinite geodesic that extends side $i$ to the circle at infinity~$\Sb$.\footnote{\,The points $P_i$ and $Q_i$ in this paper and~\cite{BS79,KU17} are denoted by $a_i$ and $b_{i-1}$, respectively, in \cite{AF91}.} The order of endpoints on $\Sb$ is the following:
\[ P_1, Q_1, P_2, Q_2, \ldots, Q_{8g-4}. \]
The identification of the sides of $\F$ is given by the side pairing rule
\[ \sigma(i) := \left\{ \begin{array}{ll}
    4g-i \bmod (8g-4) & \text{ if $i$ is odd} \\
    2-i \bmod (8g-4) & \text{ if $i$ is even}
\end{array} \right. \]
The generators $T_i$ of $\G$ associated to this fundamental domain are M\"obius transformations satisfying the following properties:
\begin{align}
    T_{\sigma(i)}T_i&=\Id\label{r11}\\
    T_i(V_i)&=V_{\rho(i)}, \text{ where } \rho(i)=\sigma(i)+1\label{r12}\\
    T_{\rho^3(i)}T_{\rho^2(i)}T_{\rho(i)}T_i&=\Id\label{r13}
\end{align}

According to B. Weiss \cite{W92}, the existence of such a fundamental polygon is an old result of Dehn, Fenchel, and Nielsen, while J.~Birman and C.~Series \cite{BiS87} attribute it to Koebe~\cite{Ko29}. Adler and Flatto~\cite[Appendix A]{AF91} give a careful proof of existence and properties of the fundamental $(8g-4)$-sided polygon for any surface group $\G$ such that $\G\backslash\D$ is a compact surface of genus $g$. Notice that in general the polygon $\F$ need not be regular. If $\F$ is regular, it is the Ford fundamental domain, i.e., ${P_iQ_{i+1}}$ is the isometric circle for $T_i$, and $T_i(P_iQ_{i+1})=Q_{\sigma(i)+1}P_{\sigma(i)}$ is the isometric circle for $T_{\sigma(i)}$ so that the inside of the former isometric circle is mapped to the outside of the latter, and all internal angles of $\F$ are equal to $\frac{\pi}2$. Figure~\ref{fig disk} shows such a construction for $g=2$.
\begin{remark}
    Although the results in \cite{KU17} were obtained using the regular fundamental polygon, the general case is reduced to this specific situation by the Fenchel-Nielsen theorem  without affecting the results of \cite{KU17}, as explained in \cite[Appendix A]{AF91} and in \cite{KU17e}. More precisely, given $\G'\backslash \D$  a compact surface of genus $g>1$, there exists a Fuchsian group $\Gamma$ and (by the Fenchel-Nielsen theorem \cite{T72}) an orientation-preserving homeomorphism $h$ from $\bar\D$ onto $\bar\D$ such that
\begin{enumerate}
    \item $\G\backslash \D$ is a compact surface of the same genus $g$;
    \item $\Gamma$ has a fundamental domain $\mathcal F$ given by a regular $(8g-4)$-sided polygon;
    \item $\G'=h\circ \G\circ h^{-1}$;
    \item $\G'$ has a fundamental domain $\mathcal F'$ given by an $(8g-4)$-sided polygon whose sides are produced by the geodesics $h(P_i)h(Q_{i+1})$ and identified by $T'_i=h\circ T_i\circ h^{-1}$, where $\{T_i\}$ are generators of $\G$ identifying the sides of $\mathcal F$. 
\end{enumerate}
In this paper we use the regular polygon in some proofs for convenience. 
\end{remark}

A \emph{cross-section} $C$ for the geodesic flow $\{\varphi^t\}$ is a subset of the unit tangent bundle $SM$ visited by (almost) every geodesic infinitely often both in the future and in the past. It is well-known that the geodesic flow can be represented as a {\em special flow} on the space
\[ {C}^{h}=\{\, (v,s) : v\in C, 0\leq s\leq h(v) \,\}. \]
It is given by the formula $\varphi^t(v,s)=(v, s+t)$ with the identification $(v,h(v))=(R(v),0)$, where the ceiling function $h:C\to\R$ is the {\em time of the first return} of the geodesic defined by $v$ to $C$, and $R:C\to C$ given by $R(v)=\varphi^{h(v)}(v)$ is the {\em first return map} of the geodesic to $C$.

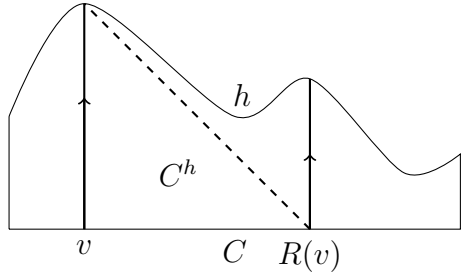
\begin{figure}[thb]
\begin{tikzpicture}
	\draw (0,0) -- (0,1.5) plot [smooth] coordinates {(0,1.5) (1,3) (3,1.5) (4,2) (5.25,0.75) (6,1)} -- (6,0) -- (0,0);
	\draw [thick] (1,0) node [below] {$v$} edge [->] (1,1.75) edge [-] (1,3);
	\draw [thick] (4,0) node [below] {$R(v)$} edge [->] (4,1) edge [-] (4,2);
	\draw [thick,dashed] (4,0) -- (1,3);
	\draw (3,0) node [below] {$C$};
	\draw (2.25,0.75) node {$C^h$};
	\draw (3.1,1.5) node [above] {$h$};
\end{tikzpicture}\vspace*{-1em}
\caption{Geodesic flow is a special flow}
\end{figure}

Let $\mathcal N$ be a finite or countable alphabet, ${\mathcal N}^{\Z}=\{\, s = \{n_i\}_{i\in \Z} : n_i \in \mathcal N \,\}$ be the space of all bi-infinite sequences endowed with the Tikhonov (product) topology. Let
\[ \sigma:{\mathcal N}^\Z\to {\mathcal N}^\Z\text{ defined by }(\sigma s)_i=n_{i+1} \]
be the left shift map, and $\Lambda\subset {\mathcal N}^\Z$ be a closed $\sigma$-invariant subset. Then $ (\Lambda,\sigma)$ is called a {\em symbolic dynamical system}. There are some important classes of such dynamical systems. The space $({\mathcal N}^\Z, \sigma)$ is called the {\em full shift} (or the {\em topological Bernoulli shift}). If the space $\Lambda$ is given by a set of simple transition rules which can be described with the help of a matrix consisting of zeros and ones, we say that $ (\Lambda,\sigma)$ is a {\em one-step topological Markov chain} or simply a {\em topological Markov chain} (sometimes $(\Lambda,\sigma)$ is also called a {\em subshift of finite type}). A factor of a topological Markov chain is called a {\em sofic shift}. For the definitions, see~\cite[Sec.~1.9]{KH}.

\smallskip

In order to represent the geodesic flow as a special flow over a symbolic dynamical system, one needs to choose an appropriate cross-section $C$ and code it, i.e., find an appropriate symbolic dynamical system $ (\Lambda,\sigma)$ and a continuous surjective map $\Cod: \Lambda\to C$ (in some cases the actual domain of $\Cod$ is $\Lambda$ except a finite or countable set of excluded sequences) defined such that the diagram
\[ \commutativeDiagram{\Lambda}{\sigma}{\Lambda}{\Cod}{\Cod}{C}{R}{C} \]
is commutative. We can then talk about {\em coding sequences} for geodesics defined up to a shift that corresponds to a return of the geodesic to the cross-section $C$. Notice that usually the coding map is not injective but only finite-to-one (see, e.g.,~\cite[\S 3.2 and \S 5]{A98} and Examples \renewcommand{\thesubsection}{\arabic{subsection}}\ref{mult geom} and~\ref{mult arith} in Section~\ref{sec examples} of this paper).

Slightly paraphrased, Adler and Flatto's method of representing the geodesic flow on $M$ as a special flow is the following. Consider the set $C_G$ of unit tangent vectors $(z,\zeta)\in S\D$ based on the boundary of $\F$ and pointed inside $\F$.\footnote{\,Adler and Flatto \cite[p.~240]{AF91} used unit tangent vectors pointed out of $\F$, but the results are easily transferable.} Every geodesic $\gamma$ in $\D$ is equivalent to one intersecting the fundamental domain $\F$ and thus is comprised of countably many segments in $ \F$. More precisely, if we start with a segment $\g\cap\F$ which enters $\F$ through side $j$ and exits $\F$ through side $i$, then $T_i\g\cap\F$ is the next segment in $\F$, and $T_{j}\g\cap\F$ is the previous segment.
It is clear that the canonical projection of $\g$ to $SM$, $\pi(\g)$, visits $C_G$ infinitely often, hence $C_G$ is a cross-section which we call {\em the geometric cross-section}. The geodesic $\pi(\g)$ can be coded by a bi-infinite sequence of generators of $\G$ identifying the sides of $\F$, as explained below. This method of coding goes back to Morse; for more details see~\cite{K96}.

The set $\Omega_G$ of oriented  geodesics in $\D$ tangent to the vectors in $C_G$ coincides with the set of geodesics in $\D$ intersecting $\F$; this is depicted in Fig.~\ref{fig rainbow CL} in coordinates $(u,w)\in \Sb\times\Sb, u\neq w$, where $u$ is the beginning  and $w$ is the end of a geodesic. The coordinates of the ``vertices'' of $\Omega_G$ are $(P_j,Q_{j+1})$ (the upper part) and $(Q_j,P_j)$ (the lower part). Let $\Gs_i$ be the ``curvilinear horizontal slice'' comprised of points $(u,w)$ such that geodesics $uw$ exit $\F$ through side $i$; this slice is in the horizontal strip between $P_i$ and $Q_{i+1}$. The map $F_G(u,w)$ piecewise transforms variables by the same M\"obius transformations that identify the sides of $\F$, that is,
\begin{equation}
    F_G(u,w)=(T_iu, T_iw) \quad\text{if } (u,w)\in \Gs_i,
\end{equation}
and $\Gs_i$ is mapped to the ``curvilineal vertical slice'' belonging to the vertical strip between $P_{\sigma(i)}$ and $Q_{\sigma(i)+1}$. $F_G$ is a bijection of $\Omega_G$. Due to the symmetry of the fundamental polygon $\F$, the ``curvilinear horizontal (vertical) slices'' are congruent to each other by a Euclidean translation. 

For a geodesic $\g$, the {\em geometric coding sequence}
\[ 
	[\gamma]_G = \sequence[G]{\ldots,n_{-2},n_{-1},n_0,n_1,n_2,\ldots}
\] 
is such that $\sigma(n_k)$ is the side of $\F$ through which the geodesic $F_G^k\g$ exits $\F$. By construction, the left shift in the space of geometric coding sequences corresponds to the map $F_G$, and the geodesic flow $\{\varphi^t\}$ becomes a special flow over a symbolic dynamical system.

\begin{figure}[hbt]
\begin{tikzpicture}
	\draw (0,0) node [left=1em] {\includegraphics[width=0.45\textwidth]{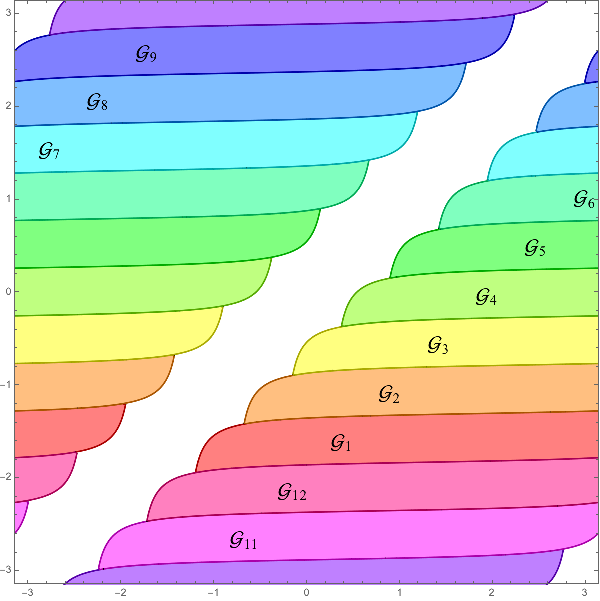}};
	\draw (0,0) node [above] {\;$\stackrel{F_G}{\longrightarrow}$};
	\draw (0,0) node [right=1em] {\includegraphics[width=0.45\textwidth]{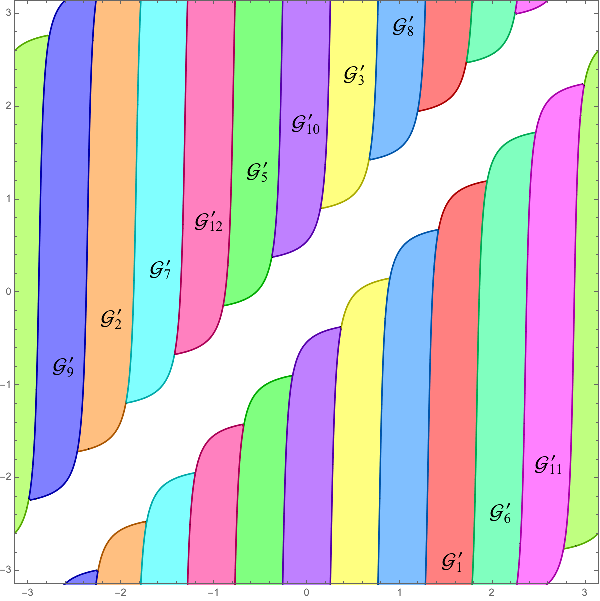}};
\end{tikzpicture}\vspace*{-1em}
\caption{Curvilinear set $\Omega_G$ and its image for $g=2$. Here $\Gs'_i = F_G(\Gs_i)$.}
\label{fig rainbow CL}
\end{figure}

The set of geometric coding sequences is natural to consider, but it is never Markov. In order to obtain a special flow over a Markov chain, Adler and Flatto replaced the curvilinear boundary of the set $\Omega_G$ by polygonal lines, obtaining what they call a ``rectilinear'' set $\Omega_\P$ and defining an auxiliary map $F_\P$ mapping any horizontal line into a part of a horizontal line. We denote by $F_G:\Omega_G \to \Omega_G$ and $F_\P:\Omega_\P \to \Omega_\P$, respectively, their \emph{curvilinear} and \emph{rectilinear} transformations.\footnote{Starting in Section~\ref{sec reduction}, we deal with a family of generalized rectilinear maps $F_\A$, and the notations $F_\P$ and $\Omega_\P$ are in conformance with those.}
They prove that these maps are conjugate and that the rectilinear map is sofic.

We should mention the relation between \cite{AF91} and the series of papers by C.~Series \cite{S81, S81', S86} where she also studies the problem of representing the geodesic flow by symbolic systems following the fundamental paper \cite{BS79}. Adler and Flatto acknowledge that they reach essentially the same mathematical conclusions. However, the use of the $(8g-4)$-sided fundamental polygon $\F$ in \cite{AF91} makes their exposition much more accessible and allows for extensions of their results.

For the modular surface, the rectilinear map is related to a Gauss-like $1$-dimensional map and continued fractions, and for the compact surface case, to the boundary map considered previously by Bowen and Series~\cite{BS79}. In fact, it is exactly the natural extension of the corresponding boundary map. The authors of~\cite{KU12} studied so-called $(a,b)$-continued fractions and corresponding boundary maps for coding of geodesics on the modular surface. Their method is based on  the results of ~\cite{KU10} establishing finite rectangular structure of the attractors of the associated natural extension maps and the corresponding ``reduction theory.''

In this paper we use similar approach for the compact surface case based on the results of~\cite {KU17}. Here again we consider a family of boundary maps which give rise to a family of arithmetic cross-sections parametrized by attractors of the corresponding natural extension maps---\emph{rectilinear} maps which we study on their own,  thus representing the geodesic flow as a special flow over a symbolic system of coding sequences.
The finite rectangular structure of these attractors and the corresponding ``reduction theory'' plays an essential role in our approach. We show that for some of the arithmetic cross-sections the resulting symbolic system turns out to be sofic, thus extending and in some ways simplifying the Adler-Flatto results.

The paper is organized as follows. Section~\ref{sec reduction} recalls the definitions of the relevant maps as well as previous results on attractors. Section~\ref{sec conjugacy} describes conjugacy between curvilinear and rectilinear maps, extending Adler-Flatto's phenomenon of ``bulges mapping to corners" to a broader class of boundary maps. In Section~\ref{sec cross-section}, we define a class of arithmetic cross-section for the geodesic flow on $M = \G\backslash\D$, and Section~\ref{sec coding} describes the arithmetic coding of geodesics using these cross-sections. Although the arithmetic cross-sections we consider actually coincide with the geometric one, the symbolic dynamical systems are different (some examples are discussed in Section~\ref{sec examples}). While the set of geometric coding sequences is never Markov, for some choice of parameters the arithmetic ones are. This is discussed in Sections~\ref{sec Markov} and~\ref{sec examples of Markov}. In Section~\ref{dual codes} we discuss dual codes.  In Section~\ref{sec entropy}, we apply results of~\cite{KU17} to obtain explicit formulas for invariant measures and calculate the measure-theoretic entropy.

 \section{The reduction procedure} \label{sec reduction}

Let $f_\A $ be a generalized Bowen-Series boundary map studied in~\cite{KU17} and defined by the formula 
\begin{equation}
    f_\A (x)=T_i(x) \quad\text{if } x \in [A_i,A_{i+1}),
\end{equation}
where
\[ \A=\{A_1,\dots,A_{8g-4}\},\quad A_i\in (P_i,Q_i), \]
and let $F_\A $ be the corresponding two-dimensional map: 
\begin{equation} \label{FA}
    F_\A (x,y)=(T_i(x),T_i(y)) \quad\text{if } y \in [A_i,A_{i+1}).
\end{equation}
The map $F_\A $ is a \emph{natural extension} of the boundary map $f_\A $. Setting $\Delta = \{\, (x,y) \in \Sb \times \Sb : x = y \,\}$, we have $F_\A$ as a map on $\Sb \times \Sb \setminus \Delta$. If we identify a geodesic in $\D$ from $u$ to $w$ with a point in $\Sb\times\Sb \setminus \Delta$, $F_\A$ may also be considered as a map on geodesics.

Adler and Flatto~\cite{AF91} studied the partition 
\[ \P = \{ P_1, \ldots, P_{8g-4} \}, \]
which was previously considered by Bowen and Series~\cite{BS79}, and their ``rectilinear map'' ${T}_{R}$ is exactly our $F_\P$.

A key ingredient in analyzing the map $F_\A $ is what we call the \emph{cycle property} of the partition points $\{A_1,\dots,A_{8g-4}\}$. Such a property refers to the structure of the orbits of each $A_i$ that one can construct by tracking the two images $T_iA_i$ and $T_{i-1}A_i$ of these points of discontinuity of the map $f_\A $. It happens that some forward iterates of these two images $T_iA_i$ and $T_{i-1}A_i$ under $f_\A $ coincide~\cite [Theorem 1.2] {KU17}.

If a cycle closes up after one iteration, that is,
\begin{equation}
    \label{eq:sc} f_\A (T_iA_i)=f_\A (T_{i-1}A_i),
\end{equation}
we say that the point $A_i$ satisfies the \emph{short cycle property}. We say a partition $\A$ has the \emph{short cycle property} if each $A_i \in \A$ has the short cycle property.

As proved in~\cite[Theorems 1.3 and 2.1]{KU17}, if the partition $\A$ has a short cycle property or if $\A=\P$, then $F_\A$ has a {\em global attractor}  
\[ \Omega_\A = \bigcap_{n=0}^\infty F_\A^n( \Sb \times \Sb \setminus \Delta ), \]
which is shown to have {\em finite rectangular structure}, i.e., it is bounded by non-decreasing step-functions with a finite number of steps (see Figure~\ref{fig attractor}) on which $F_\A$ is essentially bijective, and almost every point $(u,w) \in \Sb \times \Sb \setminus \Delta$ is mapped to $\Omega_\A$ after finitely many iterations of $F_\A$.

Explicitly, the $y$-levels of the upper and lower connected component of $\Omega_{\A}$ are, respectively,
\begin{equation}
    \label{B and C} B_i := T_{\sigma(i-1)}A_{\sigma(i-1)} \qquad\text{and}\qquad C_i := T_{\sigma(i+1)}A_{\sigma(i+1)+1}.
\end{equation}
The $x$-levels in this case are the same as for the Bowen-Series map $F_{\P}$, and the set $\Omega_{\A}$ is determined by the corner points located in the strip
$\{\, (x, y) \in \Sb \times \Sb : y \in [A_i, A_{i+1}) \,\}$
with coordinates 
\[ (P_i,B_i) \quad\text{(upper part)} \qquad\text{and}\qquad (Q_{i+1},C_i) \quad\text{(lower part)}. \]
See Figure~\ref{fig attractor} for an example of such an attractor.

\begin{figure}
	\includegraphics{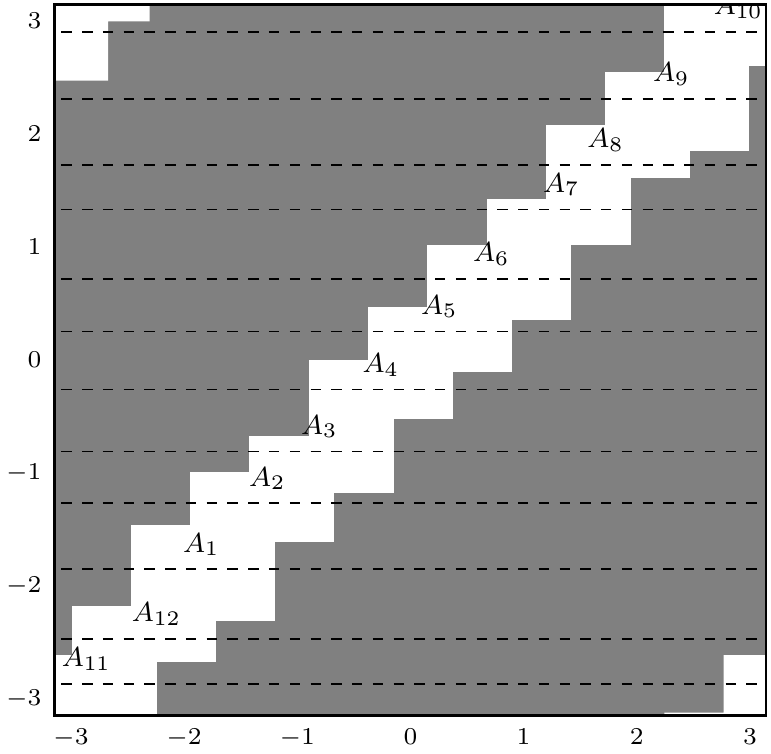}
	\caption{Attractor $\Omega_\A$ for generic $F_\A$ with $g=2$.}
	\label{fig attractor}
\end{figure}

\section{Conjugacy between \texorpdfstring{$F_\A$ and $F_G$}{arithmetic and geometric maps}} \label{sec conjugacy}

The following useful property of $T_i$ will be referenced throughout the paper.
\begin{prop}
    [Proposition~{2}.{2} in \cite{KU17} and Theorem~3.4 in \cite{AF91}] \label{T images} $T_i$ maps the points $P_{i-1}$, $P_i$, $Q_i$, $P_{i+1}$, $Q_{i+1}$, $Q_{i+2}$ respectively to $P_{\sigma(i)+1}$, $Q_{\sigma(i)+1}$, $Q_{\sigma(i)+2}$, $P_{\sigma(i)-1}$, $P_{\sigma(i)}$, $Q_{\sigma(i)}$.
\end{prop}
We will also make use of the map that we call here $U_i$, which is used in both \cite{AF91} and \cite{KU17} without its own notation. The following are all equivalent definitions:
\begin{equation}
    \label{eq:Ui}
    \begin{array}{l@{\;}l}
        U_i &:= (T_{\sigma(i)+1}T_i)^{-1} = T_i^{-1} T_{\sigma(i)+1}^{-1} = T_{\sigma(i)}T_{\tau(i-1)} \\
        &\;= (T_{\sigma(i-1)-1}T_{i-1})^{-1} = T_{i-1}^{-1} T_{\sigma(i-1)-1}^{-1} = T_{\sigma(i-1)} T_{\tau(i)},
    \end{array}
\end{equation}
where
\begin{equation}
    \tau(i) := i+(4g-2) \quad\bmod (8g-4).
\end{equation}
Most importantly, the map $U_i$ sends $\F$ to the ``corner image'' $\F_i := U_i \F$ that touches $\F$ only at the vertex $V_i$ (e.g., Figure~\ref{fig first return} on page \pageref{fig first return} shows $\F_2$ in gray). Note that $\{T_i\}$ and $\{U_i\}$ are both closed under inversion; specifically, $T_i^{-1} = T_{\sigma(i)}$ via \eqref{r11} and
\[ U_i^{-1} = U_{\tau(i)}, \]
which can be shown using \eqref{r11} and Lemma~\ref{tau properties} below. Additionally, $U_i^{-1}$ maps $A_i$ to the end of the cycle (\ref{eq:sc}), hence the two main expressions in (\ref{eq:Ui}).

\medskip Lastly, we will need the following properties of $\tau$ and $\sigma$:
\begin{samepage}
    \begin{lem}
        \label{tau properties} If $\theta(i) := \sigma(i)-1$ and $\rho(i) := \sigma(i)+1$, then
        \begin{enumerate}
            \item $\theta^2(i) = \rho^2(i) = \tau(i)$, \item $\rho(i+1) = \theta(i-1) = \tau(\sigma(i))$, \item $\sigma$ and $\tau$ commute.
        \end{enumerate}
    \end{lem}
\end{samepage}

\begin{proof}
    These can be verified by direct calculation. As one instance, we show that for even~$i$, $\rho(i+1) = \tau(\sigma(i))$. In this case
    \begin{align*}
        \sigma(i) &= 2-i \qquad\text{($i$ is even)} \\
        \tau(\sigma(i)) &= 2-i+(4g-2) = 4g-i \\
        \sigma(i+1) &= 4g-(i+1) \qquad\text{($i+1$ is odd)} \\
        \rho(i+1) &= 4g-(i+1)+1 = 4g-i
    \end{align*}
    The other equalities are shown similarly.
\end{proof}
Note that $\tau(i-1) = \tau(i)-1$ because $\tau$ is a shift; these two forms are used interchangeably in several equations. In general, $\sigma(i-1) \ne \sigma(i)-1$.

\medskip For the case $\A = \P$, Adler and Flatto \cite[Sec.~5]{AF91} explicitly describe a~conjugacy between $F_G$ and $F_\P$. They define the function $\Phi : \Omega_G \to \Omega_\P$ given by
\[ \Phi = \left\{ \begin{array}{ll}
    \Id & \text{on } \Os \\
    U_{\tau(i)+1} & \text{on } \Bs_i,
\end{array} \right. \]
where $\Os = \Omega_G \cap \Omega_\P$ and the set $\Bs_i$ is the ``bulge'' that is the part of $\Omega_G \setminus \Omega_\P$ with $w \in [P_i,P_{i+1}]$. They then show that the diagram
\[ \commutativeDiagram{\Omega_G}{F_G}{\Omega_G}{\Phi}{\Phi}{\Omega_\P}{F_\P}{\Omega_\P} \]
commutes. Note that for $\A = \P$, the bulges (comprising all points in \mbox{$\Omega_G \setminus \Omega_\P$}) are affixed only to the lower part of the rectilinear set $\Omega_\P$.

For the case of a generic $\A$ with the short cycle property, the set \mbox{$\Omega_G \setminus \Omega_\A$} has pieces affixed to both the upper and lower parts of $\Omega_\A$. Thus we must define both upper and lower bulges: 
\begin{equation} \label{def of bulges}
    \begin{split}
        \text{\em lower bulge } \Bs_i &= \{(u,w) \in \Omega_G \setminus \Omega_\A : u \in [Q_{i+1},Q_{i+2}] \}; \\
        \text{\em upper bulge } \Bs^i &= \{(u,w) \in \Omega_G \setminus \Omega_\A : u \in [P_{i-1},P_i]\}.
    \end{split}
\end{equation}
We also define corners, which comprise $\Omega_\A \setminus \Omega_G$:
\begin{equation} \label{def of corners}
    \begin{split}
        \text{\em lower corner } \Cs_i &= \{(u,w) \in \Omega_\A \setminus \Omega_G : u \in [Q_{i+1},Q_{i+2}] \}; \hspace*{2.55em} \\
        \text{\em upper corner } \Cs^i &= \{(u,w) \in \Omega_\A \setminus \Omega_G : u \in [P_{i-1},P_i]\}.
    \end{split}
\end{equation}
Figure~\ref{fig bulges and corners} shows all four of these sets for a single~$i$.
\begin{remark}
    The curvilinear and rectilinear sets of Adler and Flatto include one boundary but not the other (see \cite[Fig.\;4.7, 5.1]{AF91}). Our convention is that $\Omega_G$ and $\Omega_\A$ are closed, so for us each bulge includes its curved boundary but not its straight boundaries and each corner includes its straight boundaries but not its curved boundary. This does not affect the overall dynamics in any significant way.
\end{remark}

\begin{figure}[ht]
	\includegraphics{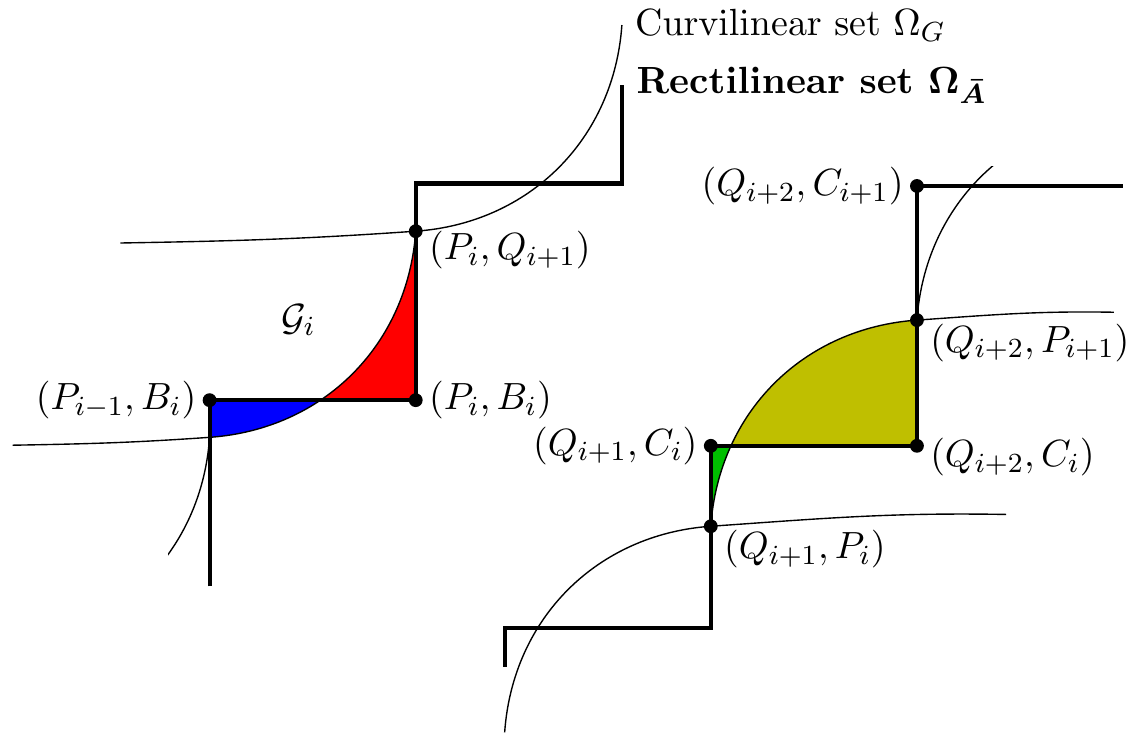}
	\caption{Bulges $\Bs^i$ in blue and $\Bs_i$ in gold; corners $\Cs^i$ in red and $\Cs_i$ in green.}
	\label{fig bulges and corners}
\end{figure}

Now let $\Os := \Omega_G \cap \Omega_\A$ and define the map $\Phi$ with domain $\Omega_G$ as
\begin{equation} \label{Phi}
    \Phi = \left\{
    \begin{array}{ll}
        \Id & \text{on } \Os \\
        U_{\tau(i)+1} & \text{on } \Bs_i \\
        U_{\tau(i)} & \text{on } \Bs^i.
    \end{array}
    \right.
\end{equation}

\begin{prop}
    \label{bulges map to corners} Let $\A$ have the short cycle property. Then the map $\Phi$ is a bijection from $\Omega_G$ to $\Omega_\A$. Specifically, $\Phi(\Bs_i) = \Cs^{\tau(i)+1}$ and $\Phi(\Bs^i) = \Cs_{\tau(i)-1}$.
\end{prop}

\begin{proof}
    All the sets $\Bs_i, \Bs^i, \Cs_i, \Cs^i$ are bounded by one horizontal line segment, one vertical line segment, and one curved segment that is part of $\partial \Omega_G$. Since each $U_j$ is a M\"obius transformation, we need only to show that these boundaries are mapped accordingly. The boundaries of $\Bs_i$ are the vertical segment $\{Q_{i+2}\} \times [C_i,P_{i+1}]$, part of the horizontal segment $[Q_{i+1},Q_{i+2}] \times \{C_i\}$, and the part of $\partial \Omega_G$ connecting $(Q_{i+1},P_i)$ to $(Q_{i+2},P_{i+1})$. By definition, $\Phi(\Bs_i) = U_{\tau(i)+1}(\Bs_i)$. Recall that $U_{\tau(i)+1} = T_{\sigma(i)-1}T_i$. From Proposition~\ref{T images}, we have
    \begin{align*}
        T_i P_i &= Q_{\sigma(i)+1}, & T_i P_{i+1} &= P_{\sigma(i)-1}, & T_i Q_{i+1} &= P_{\sigma(i)}, & T_i Q_{i+2} &= Q_{\sigma(i)}.
    \end{align*}
    Applying $T_{\sigma(i)-1}$ to both sides and using Lemma~\ref{tau properties}, we get
%    \begin{align*}
%        U_{\tau(i)+1} P_i &= T_{\sigma(i)-1} Q_{\sigma(i)+1} = Q_{\tau(i)+1}, & U_{\tau(i)+1} P_{i+1} &= T_{\sigma(i)-1} P_{\sigma(i)-1} = Q_{\tau(i)+2}, \\
%        U_{\tau(i)+1} Q_{i+1} &= T_{\sigma(i)-1} P_{\sigma(i)} = P_{\tau(i)}, & U_{\tau(i)+1} Q_{i+2} &= T_{\sigma(i)-1} Q_{\sigma(i)} = P_{\tau(i)+1}.
%    \end{align*}
    \begin{align*}
        U_{\tau(i)+1} P_i &= T_{\sigma(i)-1} Q_{\sigma(i)+1} = Q_{\tau(i)+1}, \\
        U_{\tau(i)+1} P_{i+1} &= T_{\sigma(i)-1} P_{\sigma(i)-1} = Q_{\tau(i)+2},\\
        U_{\tau(i)+1} Q_{i+1} &= T_{\sigma(i)-1} P_{\sigma(i)} = P_{\tau(i)}, \\
        U_{\tau(i)+1} Q_{i+2} &= T_{\sigma(i)-1} Q_{\sigma(i)} = P_{\tau(i)+1}.
    \end{align*}
    Additionally, by \eqref{B and C},
    \begin{equation}
        \label{UBC} U_{\tau(i)+1} C_i = (T_{\sigma(i+1)+1}T_{i+1}) (T_{\sigma(i+1)}A_{\tau\sigma(i)}) = T_{\sigma(\tau(i))} A_{\sigma(\tau(i))} = B_{\tau(i)+1}.
    \end{equation}
    Therefore
    \begin{align*}
        U_{\tau(i)+1}\big( \{Q_{i+2}\} \times [C_i,P_{i+1}] \big) &= \{P_{\tau(i)+1}\} \times [B_{\tau(i)+1}, Q_{\tau(i)+2}], \\
        U_{\tau(i)+1}\big( [Q_{i+1},Q_{i+2}] \times \{C_i\} \big) &= [P_{\tau(i)},P_{\tau(i)+1}] \times \{B_{\tau(i)+1}\}, \\
        U_{\tau(i)+1} (Q_{i+1},P_i) &= (P_{\tau(i)}, Q_{\tau(i)+1}), \\
        U_{\tau(i)+1} (Q_{i+2},P_{i+1}) &= (P_{\tau(i)+1}, Q_{\tau(i)+2}).
    \end{align*}
    The corner $\Cs^{\tau(i)+1}$ is exactly the set bounded by the vertical segment $\{P_{\tau(i)+1}\} \times [B_{\tau(i)+1}$, $Q_{\tau(i)+2}]$, part of the horizontal segment $[P_{\tau(i)},P_{\tau(i)+1}] \times \{B_{\tau(i)+1}\}$, and part of segment of $\partial \Omega_G$ connecting $(P_{\tau(i)}, Q_{\tau(i)+1})$ to $(P_{\tau(i)+1}, Q_{\tau(i)+2})$. Thus $U_{\tau(i)+1}(\Bs_i) = \Cs^{\tau(i)+1}$. A similar argument shows $U_{\tau(i)}(\Bs^i) = \Cs_{\tau(i)-1}$. Taking $\Phi(\Os) = \Id(\Os) = \Os$ together with $\Phi(\Bs_i) = \Cs^{\tau(i)+1}$ and $\Phi(\Bs^i) = \Cs_{\tau(i)-1}$, we have that $\Phi(\Omega_G) = \Omega_\A$.
\end{proof}

\begin{defn}
    Let $u,w\in \Sb=\partial\D$, $u\neq w$. An oriented geodesic in $\D$ from $u$ to $w$ is called \emph{$\A$-reduced} if $(u,w)\in\Omega_\A$.
\end{defn}

\begin{cor}
    \label{anti corner} If a geodesic $\g = uw$ intersects $\F$, then either $\g$ is $\A$-reduced or $U_j \g$ is $\A$-reduced, where $j=\tau(i)$ if $(u,w) \in \Bs^i$ and $j=\tau(i)+1$ if $(u,w) \in \Bs_i$.
\end{cor}

\begin{proof}
    The geodesic $\g = uw$ intersecting $\F$ is equivalent to $(u,w)$ being in the set $\Omega_G$. If $(u,w) \in \Os = \Omega_G \cap \Omega_\A$, then $\g$ is $\A$-reduced as well. If not, then $(u,w)$ is in some upper or lower bulge since the bulges comprise all of $\Omega_G \setminus \Omega_\A$. If $(u,w) \in \Bs_i$, then $U_{\tau(i)+1}(u,w) \in \Cs^{\tau(i)+1} \subset \Omega_\A$ and so $U_{\tau(i)+1}\g$ is $\A$-reduced. If $(u,w) \in \Bs^i$, then $U_{\tau(i)}(u,w) \in \Cs_{\tau(i)-1} \subset \Omega_\A$ and so $U_{\tau(i)}\g$ is $\A$-reduced.  It follows that $\g$ and $U_j\g$ cannot be $\A$-reduced simultaneously.
\end{proof}

\begin{remark}
    \label{two corner images} Note that although the geodesic $\g$ might intersect two different corner images, the point $(u,w)$ cannot be in two different bulges simultaneously. The index of the bulge $\Bs^i$ or $\Bs_i$ determines a specific $j$ for which $U_j \g$ is $\A$-reduced, as stated in Corollary~\ref{anti corner}.
\end{remark}

\begin{cor}
    \label{corner} If a geodesic $\g=uw$ is $\A$-reduced, then either $\g$ intersects $\F$, or $\g$ intersects $\F_j = U_j(\F)$, where $j=i$ if $(u,w) \in \Cs^i$ and $j=i+1$ if $(u,w) \in \Cs_i$.
\end{cor}

\begin{proof}
    The geodesic $\g$ is $\A$-reduced means that $(u,w)\in\Omega_{\A}$. If $(u,w) \in \Os = \Omega_G \cap \Omega_\A$, then $\g$ intersects $\F$. If not, then $(u,w)$ is in some upper or lower corner. If $(u,w)\in \Cs_i$, then $U^{-1}_i(u,w)\in \Bs_{\tau(i)-1}\subset \Omega_G$, so $U^{-1}_i(\g)$ intersects $\F$, or $\g$ intersects $U_i(\F)=\F_i$.

If $(u,w)\in \Cs^i$, then $U^{-1}_{i+1}(u,w)\in \Bs^{\tau(i)+1}$, so $U^{-1}_{i+1}(\g)$ intersects $\F$, or $\g$ intersects $U_{i+1}(\F)=\F_{j+1}$.
\end{proof}

\begin{remark}
    \label{two corner images 2} Here again, although the geodesic $\g$ might intersect two different corner images, the index of the corner $\Cs^i$ or $\Cs_i$ determines a specific $j$ for which $\g$ intersects $\F_j$, as stated in Corollary~\ref{corner}.
\end{remark}

\begin{thm}
    \label{conjugacy} Let $\A$ have the short cycle property, and let $\Phi : \Omega_G \to \Omega_\A$ be as in \eqref{Phi}. Then $\Phi$ is a conjugacy between $F_G$ and $F_\A$. That is, the following diagram commutes:
    \begin{equation}
        \label{cd:conjugacy} \commutativeDiagram{\Omega_G}{F_G}{\Omega_G}{\Phi}{\Phi}{\Omega_\A}{F_\A}{\Omega_\A}.
    \end{equation}
\end{thm}

\begin{proof}
    The diagram commuting is equivalent to
    \begin{equation}
        \label{commutes} F_\A^{-1} \circ \Phi \circ F_G \circ \Phi^{-1} = \Id.
    \end{equation}
    Let $(u,w)$ be any point in $\Omega_\A$; we aim to prove that \eqref{commutes} is true for $(u,w)$ by building out the diagram
    \[ \commutativeDiagram{\quad}{F_G}{\quad}{\Phi}{\Phi}{(u,w)}{F_\A}{\quad}. \]
    If $(u,w) \in \Os = \Omega_G \cap \Omega_\A$, then $\Phi^{-1}(u,w) = (u,w)$. Let $i$ be the side through which the geodesic $\g = uw$ exits $\F$; then $F_G(u,w) = T_i(u,w)$. At this point \eqref{commutes} is
    \[ F_\A^{-1} \circ \Phi \circ (T_i) \circ (\Id)^{-1} = \Id, \]
    corresponding to
    \[ \commutativeDiagram{(u,w)}{F_G = T_i}{\quad}{\Phi = \Id}{\Phi}{(u,w)}{F_\A}{\quad}. \]
    Since $\g = uw$ exits $\F$ through side $i$, the point $w$ must be in $[P_i,Q_{i+1}]$, which means that $F_\A$ must act by $T_{i-1}$, $T_i$, or $T_{i+1}$ (see Figure~\ref{fig side exit}).
    \begin{figure}[ht]
        \includegraphics{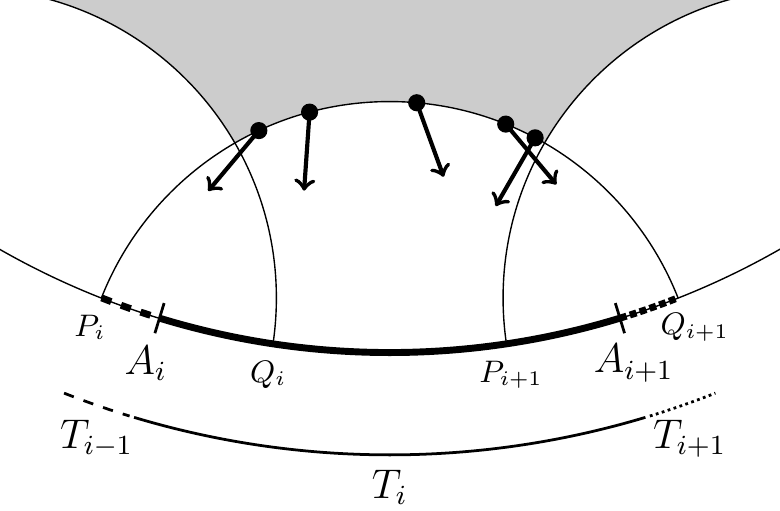}
    	\caption{Action of $F_\A$ based on side exit.}
    	\label{fig side exit}
    \end{figure}
    
    \begin{itemize}
        \item If $F_\A = T_i$, then $F_\A \g = T_i \g$ is both $\A$-reduced and intersects $\F$. Thus $\Phi$ acts on $T_i \g$ as the identity, and \eqref{commutes} becomes
        \[ (T_i)^{-1} \circ (\Id) \circ (T_i) \circ (\Id)^{-1} = \Id. \]
        \item If $F_\A = T_{i-1}$, then \eqref{commutes} becomes
        \[ (T_{i-1})^{-1} \circ \Phi \circ (T_i) \circ (\Id)^{-1} = \Id, \]
        which is true if and only if $\Phi = T_{i-1}T_i^{-1}$, or equivalently, by (\ref{eq:Ui}), $\Phi = U_{\sigma(i-1)}$. Since we know $T_i\g$ intersects $\F$, the geodesic
        \[ F_\A\g = T_{i-1}\g = T_{i-1}T_i^{-1} (T_i\g) = U_{\sigma(i-1)}(T_i \g) \]
        must intersect the corner image $\F_{\sigma(i-1)}$. Therefore $\Phi$ does act by $U_{\sigma(i-1)}$, and \eqref{commutes} is true in this case. \item If $F_\A = T_{i+1}$, then \eqref{commutes} becomes
        \[ (T_{i+1})^{-1} \circ \Phi \circ (T_i) \circ (\Id)^{-1} = \Id, \]
        which is true if and only if $\Phi = T_{i+1}T_i^{-1}$. Similar to the previous case, the geodesic
        \[ F_\A\g = T_{i+1}\g = T_{i+1}T_i^{-1}(T_i\g) = U_{\sigma\tau(i)}(T_i\g) \]
        must intersect $F_{\sigma\tau(i)}$, so $\Phi$ acts by $U_{\sigma\tau(i)} = T_{i+1}T_i^{-1}$ and \eqref{commutes} is true.
    \end{itemize}
    If $(u,w) = \Omega_\A \setminus \Os$, then it is in an upper or lower corner. Assume $(u,w)$ is in an upper corner, and let $i$ be such that $(u,w) \in \Cs^i$. Since, by Lemma~\ref{bulges map to corners}, $U_i(\Bs_{\tau(i)-1}) = \Cs^i$, we know that $\Phi^{-1}$ acts on $\Cs^i$ by $U_i^{-1}$. Defining $\g' = u'w' = U_i^{-1}\g$, we have the partial diagram
    \[ \commutativeDiagram{(u',w')}{F_G}{\quad}{\Phi=U_i}{\Phi}{(u,w)}{F_\A}{\quad} \]
    corresponding to
    \[ F_\A^{-1} \circ \Phi \circ F_G \circ (U_i)^{-1} = \Id \]
    at this point.

	Since $\Cs^i \subset [P_{i-1},P_i] \times [B_i,Q_{i+1}]$, we have 
    \begin{equation}
        \label{Uinv i endpoints}
        \begin{split}
            u' &\in U_i^{-1}[P_{i-1},P_i] = [Q_{\tau(i)}, Q_{\tau(i)+1}], \\
            w' &\in U_i^{-1}[B_i,Q_{i+1}] = [C_{\tau(i)-1}, P_{\tau(i)}]
        \end{split}
    \end{equation}
    based on Proposition~\ref{T images} and equation~\eqref{UBC}. In most instances $\g'$ will exit $\F$ through side $\tau(i)-1$, but if $C_{\tau(i)-1}$ is very close to $P_{\tau(i)-1}$ it is possible for $\g'$ to exit through side $\tau(i)-2$ instead. Thus $F_G$ might act on $(u',w')$ by $T_{\tau(i)-1}$ or $T_{\tau(i)-2}$.

	\providecommand\Case[1]{\medskip\textbf{Case #1.}}
	The remainder of the proof is broken down into four cases:
    \begin{samepage}
        \begin{enumerate}[\qquad\bf{Case }1:]
            \item $F_\A \g = T_i \g$ and $F_G \g' = T_{\tau(i)-1} \g'$.
            \item $F_\A \g = T_i \g$ and $F_G \g' = T_{\tau(i)-2} \g'$.
            \item $F_\A \g = T_{i+1} \g$ and $F_G \g' = T_{\tau(i)-1} \g'$.
            \item $F_\A \g = T_{i+1} \g$ and $F_G \g' = T_{\tau(i)-2} \g'$.
        \end{enumerate}
    \end{samepage}
    For cases 1 and 2, $w \in [B_i,A_{i+1})$, while for 3 and 4 we have $w \in [A_{i+1},Q_{i+1}]$. For cases 1 and 3, $\g'$ exits through side $\tau(i)-1$, so the geometrically next segment of $\bar\gamma$ in $\F$ starts at $z_1 := T_{\tau(i)-1}z_1'$ and is part of the geodesic
    \begin{equation}
        \label{side exit simplifies} T_{\tau(i)-1} \g' = T_{\tau(i)-1} U_i^{-1} \g = T_{\tau(i)-1} T_{\sigma(i)+1} T_i \g = T_{\tau(i)-1} T_{\sigma(\tau(i)-1)} T_i \g = T_i \g.
    \end{equation}
    
    \Case{1} In this case, $F_\A \g = T_i \g$ and, as shown in \eqref{side exit simplifies} above, $F_G \g' = T_i \g$ as well.

	\Case{2} Here $\g'$ exits $\F$ through side $\tau(i)-2$, so the geometrically next segment of $\bar\gamma$ in $\F$ is part of $T_{\tau(i)-2} \g'$. The expression $T_{\tau(i)-2}U_i^{-1}$ does not simplify as easily as $T_{\tau(i)-1}U_i^{-1} = T_i$ did. However, we do have
    \begin{equation}
        \label{non-simplification} T_{\tau(i)-2} U_i^{-1} = T_{\tau(i)-2} T_{\rho(i)} T_i = T_{\sigma(\rho(i))-1} T_{\rho(i)} T_i = U_{\rho(j)+1}^{-1} T_j.
    \end{equation}
    The geodesic $T_{\tau(i)-2} \g'$ must intersect $\F$. This is equivalent to $U_{\rho(i)+1} T_{\tau(i)-2} \g'$ intersecting $U_{\rho(i)+1} \F = \F_{\rho(i)+1}$. Since
    \[ U_{\rho(i)+1} T_{\tau(i)-2} \g' = U_{\rho(i)+1} (T_{\tau(i)-2} U_i^{-1}) \g = U_{\rho(i)+1} (U_{\rho(i)+1}^{-1} T_i) \g = T_i \g, \]
    the geodesic $T_i \g = F_\A \g$ must intersect $\F_{\rho(i)+1}$. Thus when we ``pull back'' $F_\A \g = T_i\g$ to $\F$, we get $U_{\rho(i)+1}^{-1} T_i \g$, which by \eqref{non-simplification} is exactly $T_{\tau(i)-2}U_i^{-1} \g = F_G \g'$.

    \Case{3} Here $\g'$ exits $\F$ through side $\tau(i)-1$, so $F_G \g' = T_i \g$ by \eqref{side exit simplifies}. However, $F_\A \g = T_{i+1}\g$ instead of $T_i\g$.
        
    \begin{figure}[ht]
    	\includegraphics{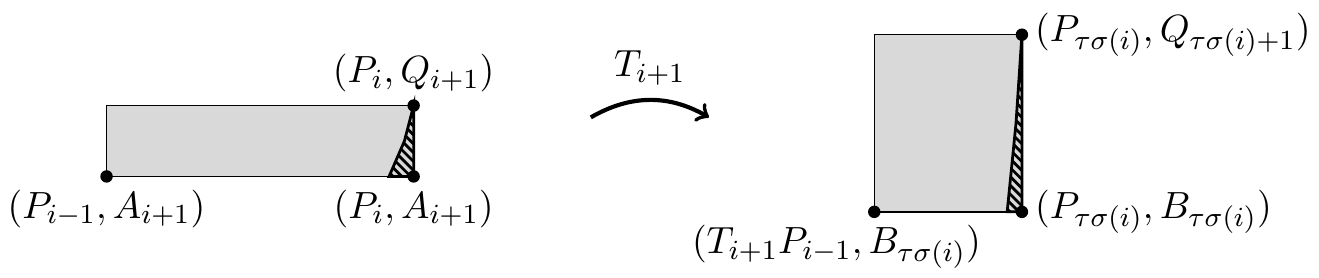}
    	\caption{The ``triangular'' region (the ``hypotenuse'' is part of $\partial\Omega_G$ so is curved) is the set of clockwise $\A$-reduced geodesics that intersect $\F_i$ but not~$\F$ and for which $F_\A = T_{i+1}$.}
    	\label{fig triangular tip}
    \end{figure}
    
    Figure~\ref{fig triangular tip} shows the upper tip of $\Cs^i$ and its image under $T_{i+1}$. The right wall $\{P_i\}\times[A_{i+1},Q_{i+1}]$ is mapped to $\{P_{\tau\sigma(i)}\} \times [B_{\tau\sigma(i)},Q_{\tau\sigma(i)}]$. As was mentioned in Section~\ref{sec introduction}, the striped region at the vertex $(P_{\tau\sigma(i)}, Q_{\tau\sigma(i)+1})$ is bounded by the curve obtained by a Euclidean translation of the boundary of the striped region at the vertex $(P_i, Q_{i+1})$, and since the map $T_{i+1}$ is hyperbolic (it expands in the $y$-direction and contracts in the $x$-direc\-tion), the image of the upper tip of the striped region at the vertex $(P_i, Q_{i+1})$ is a narrow curvilinear triangular region contained completely inside the striped region at the vertex $(P_{\tau\sigma(i)}, Q_{\tau\sigma(i)+1})$. Thus for all geodesics $\g$ in this case, $F_\A \g$ intersects $\F_{\tau\sigma(i)}$.

	When we ``pull back'' $F_\A \g = T_{i+1}\g$ to $\F$, we get
    \[ \label{non-simplification2} U_{\tau\sigma(i)}^{-1} T_{i+1} \g = T_{\sigma(\tau\sigma(i)-1)-1} T_{\tau\sigma(i)-1} T_{i+1} \g = T_i T_{\sigma(i+1)} T_{i+1} \g = T_i \g, \]
    which by \eqref{side exit simplifies} is exactly $F_G \g'$.

\Case{4} This case, in which $F_\A \g = T_{i+1}\g$ and $\g' = U_i^{-1}\g$ exits $\F$ through side $\tau(i)-2$, will be shown to be impossible.

Since the curvilinear horizontal slice $\Gs_k$, that is, the set of all geodesics exiting $\F$ through side $k$, is contained in the horizontal strip $\Sb \times [P_k,Q_{k+1}]$, the endpoint $w'$ of $\g'$ must be in $[P_{\tau(i)-2},Q_{\tau(i)-1}]$. Since $w' \in [C_{\tau(i)-1},P_{\tau(i)}]$ by \eqref{Uinv i endpoints}, we have
    \[ w' \in [P_{\tau(i)-2},Q_{\tau(i)-1}] \cap [C_{\tau(i)-1},P_{\tau(i)}] = [C_{\tau(i)-1},Q_{\tau(i)-1}] \]
    (in fact, $w'$ will be quite close to $C_{\tau(i)-1}$). This means that the endpoint $w$ of $\g = U_i \g'$ will be in the interval
    \[ I = U_i[C_{\tau(i)-1},Q_{\tau(i)-1}] = [B_i,U_iQ_{\tau(i)-1}]. \]
    We have
    \[ U_iQ_{\tau(i)-1} = T_{\sigma(i)} T_{\tau(i)-1} Q_{\tau(i)-1} = T_{\sigma(i)} Q_{\sigma(i)+3}. \]
    In order to locate $T_{\sigma(i)} Q_{\sigma(i)+3}$, we use Proposition~\ref{T images}. $T_{\sigma(i)}$ maps the geodesic $P_{\sigma(i)}Q_{\sigma(i)+1}$, which is its isometric circle, to the geodesic $Q_{i+1}P_i$. We also see that $T_{\sigma(i)} Q_{\sigma(i)+2} = Q_i$ and $T_{\sigma(i)} P_{\sigma(i)-1} = P_{i+1}$, and since $T_{\sigma(i)}$ maps the outside of $P_{\sigma(i)}Q_{\sigma(i)+1}$ to the inside of $Q_{i+1}P_i$ and preserves the order of points, we conclude that $T_{\sigma(i)}Q_{\sigma(i)+3}\in (Q_i,P_{i+1})$. Therefore $I \subset [B_i,P_{i+1}] \subset (A_i,A_{i+1})$. Thus $w \in I$ implies $F_\A \g = T_i \g$, which contradicts the assumption $F_\A\g=T_{i+1}\g$ of this case.

\medskip Cases 1--3 cover all potential ways that
    \[ \commutativeDiagram{(u',w')}{F_G}{\quad}{\Phi=U_i}{\Phi}{(u,w)}{F_\A}{\quad} \]
    can be completed, showing that \eqref{commutes} is true if $(u,w) \in \Cs^i$. Similar arguments show that \eqref{commutes} is true for $(u,w) \in \Cs_i$ as well.

Thus we have proven \eqref{commutes} for all $(u,w) \in \Omega_\A$, meaning that
    \[ \commutativeDiagram{\Omega_G}{F_G}{\Omega_G}{\Phi}{\Phi}{\Omega_\A}{F_\A}{\Omega_\A} \]
    is indeed a commutative diagram.
\end{proof}
\section{Cross-sections}\label{sec cross-section}

Based on Corollary~\ref{corner}, we introduce the notion of the {\em $\A$-cross-section point}. It is the entrance point of an $\A$-reduced geodesic $\g$ to $\F$, or, if $\g$ does not intersect $\F$, the first entrance point to $\F_j$, where $j$ is as in Corollary~\ref{corner}.

Now we define a map
\[ \varphi: \Omega_{\A}\to S\D,\quad \varphi(u,w)=(z, \zeta), \]
where $z$ is the $\A$-cross-section point on the geodesic $\g$ from $u$ to $w$ and $\zeta$ is the unit tangent vector to $\g$ at $z$. This map is clearly injective. Composed with the canonical projection $\pi : S\D \to SM$ from (\ref{projection}), we obtain a map
\begin{equation}
    \label{piphi} \pi\circ\varphi: \Omega_{\A}\to SM.
\end{equation}
The set $C_\A:=\pi\circ\varphi(\Omega_{\A})$ can be described as follows: $\pi(z,\zeta)\in C_\A$ if the geodesic~$\g$ in~$\D$ through $(z,\zeta)$ is $\A$-reduced or if $U_j \gamma$ is $\A$-reduced for $j \in \Ab$ determined in Corollary~\ref{anti corner}. It follows from Corollary~\ref{anti corner} that the map $\pi\circ\varphi$ is injective and continuous, and hence $C_\A$ is parametrized by $\Omega_{\A}$. Since $\Omega_{\A}$ is an attractor for $F_{\A}$, $C_\A$ is a cross-section for the geodesic flow $\{\varphi^t\}$; we call $C_\A$ an \emph{arithmetic cross-section}.

The \emph{geometric cross-section} $C_G$ can be described in similar terms. We define a (clearly injective) map
\[ \psi:\Omega_G\to S\D,\quad \varphi(u,w)=(z,\zeta), \]
where $z$ is the entrance point on the geodesic $\g$ from $u$ to $w$ to $\F$, and $\zeta$ is the unit tangent vector to $\g$ at $z$. Then
\[ \pi\circ\psi: \Omega_G\to SM \]
is injective, and $C_G:=\pi\circ\psi(\Omega_G)$ consists of \emph{all} $\pi(z,\zeta)$ for which $z \in \partial \F$ and~$\zeta$ points inward.

A priori, we only know that $C_\A \subset C_G$; the first return to $C_G$ is not necessarily the first return to $C_\A$. If $\G=PSL(2,\Z)$, a geodesic can return to the geometric cross-section multiple times before reaching the arithmetic one \cite{KU12}, but here the situation is simpler.

\begin{cor}
    \label{equalcross} $C_\A = C_G$.
\end{cor}

\begin{proof}
    Since $\pi\circ\varphi$, $\pi\circ\psi$, and $\Phi$ are bijections and $\Phi$ acts by elements of $\G$, the diagram
    \begin{equation}
        \label{cd:equalcross} \commutativeDiagram{\Omega_G}{\pi\circ\psi}{C_G}{\Phi}{\Id}{\Omega_\A}{\pi\circ\varphi}{C_\A}
    \end{equation}
    commutes.
    Indeed, let $\g=uw$ with $(u,w) \in \Omega_G$. If $\g$ is $\A$-reduced, then $\Phi=\Id$, $\varphi=\psi$, and we are done. If not, then for $j \in \Ab$ determined in Corollary~\ref{anti corner}, $\g' := U_j \g$ is $\A$-reduced. In this case $\Phi=U_j$, and again we have $\varphi\circ\Phi=\psi$.
\end{proof}

\begin{cor}
    \label{firstreturn} Given any tangent vector in the geometric cross-section~$C_G$, its first return to $C_G$ is also its first return to the arithmetic cross-section~$C_\A$.
\end{cor}

\begin{samepage}
\begin{proof}
    This is essentially equivalent to Theorem~\ref{conjugacy} via Corollary~\ref{equalcross}. 

Combining commutative diagrams (\ref{cd:conjugacy}) and (\ref{cd:equalcross}), we obtain a diagram
    \[ \begin{tikzpicture}[node distance=2cm]
        \node (top1) {$C_G$};
        \node [right of=top1, node distance=2.5cm] (top2) {$\Omega_G$};
        \node [right of=top2] (top3) {$\Omega_G$};
        \node [right of=top3, node distance=2.5cm] (top4) {$C_G$};
        \node [below of=top1] (bot1) {$C_\A$};
        \node [below of=top2] (bot2) {$\Omega_\A$};
        \node [below of=top3] (bot3) {$\Omega_\A$};
        \node [below of=top4] (bot4) {$C_\A$};
    	\path[->,font=\scriptsize,>=angle 90]
            (top1) edge node [above] {$(\pi\circ\psi)^{-1}$} (top2)
            	(top2) edge node [above] {$F_G$} (top3) 
    			(top3) edge node [above] {$\pi\circ\psi$} (top4)
            (top1) edge node [left]  {$\Id$} (bot1)
            (top2) edge node [left]  {$\Phi$} (bot2)
            (top3) edge node [right]  {$\Phi$} (bot3)
            (top4) edge node [right]  {$\Id$\normalsize\,.} (bot4)
            (bot1) edge node [below] {$(\pi\circ\varphi)^{-1}$} (bot2) 
            	(bot2) edge node [below] {$F_\A$} (bot3) 
    			(bot3) edge node [below] {$\pi\circ\varphi$} (bot4);
    \end{tikzpicture} \]
    The composition of the maps in the upper and lower rows are the first return maps to $C_G$ and $C_\A$, respectively. The result follows from commutativity of the diagram.
\end{proof}
\end{samepage}

Figure~\ref{fig first return} depicts the first return to the cross-section $C_\A$. Here $\g$ is $\A$-reduced but does not intersect $\F$, while  $\g' := \Phi^{-1}\g = U_2^{-1}\g$ intersects $\F$ but is not $\A$-reduced. The unit tangent vector at the point of entrance of $\g'$ to $\F$ belongs to both $C_\A$ and $C_G$. Its first return to $C_\A$ is the unit tangent vector at the point of entrance of $F_\A\g$ to $\F$, and its first return to $C_G$ is the unit tangent vector at the point of entrance of $F_G\g'$ to $\F$. Since $F_G\g'=F_\A\g$, the first returns coincide.

\begin{figure}[ht]
	\includegraphics{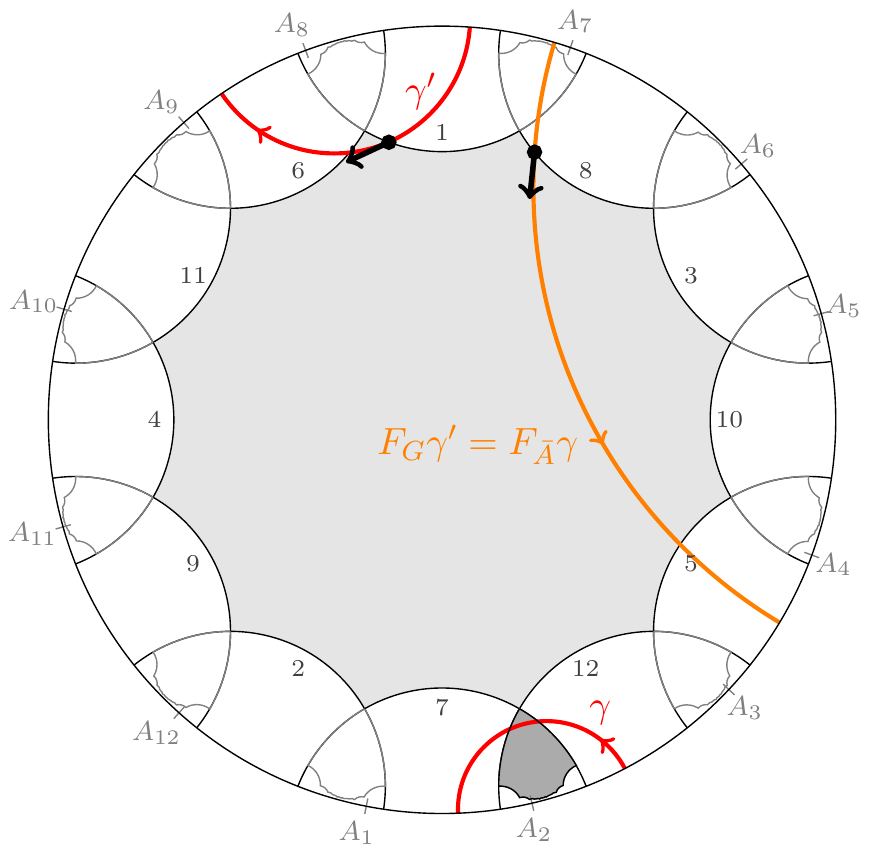}
	\caption{The first return map to the cross-section $C_\A$.}
	\label{fig first return}
\end{figure}

\section{Symbolic coding of geodesics}\label{sec coding}

In this section we describe how to code geodesics for a partition $\A$ for which $F_{\A}$ possesses an attractor $\Omega_\A$ with finite rectangular structure. A large class of examples was given in ~\cite[Theorems 1.3 and 2.1]{KU17}. A reduction algorithm described there for (almost) every geodesic $\g$ from $u$ to $w$ in $\D$ produces in finitely many steps an $\A$-reduced geodesic $\Gamma$-equivalent to $\g$, and an application of this algorithm to an $\A$-reduced geodesic produces another $\A$-reduced geodesic.

We associate to any $w_0 \in \Sb$ a sequence of symbols from the alphabet $\Ab$
\begin{equation}
    \label{future} [w_0]_\A = \sequence{ n_0, n_1, n_2,\dots },
\end{equation}
where $n_k=\sigma(i)$ if $f_\A^k(w_0)\in [A_i,A_{i+1})$ for $k\geq 0$. We call this the {\em (forward) $\A$-expansion} of $w_0$. If $\g=u_0w_0$ is an $\A$-reduced geodesic, we also call this sequence the {\em future} of $\g$.

By successive application of the map $F_\A$, we obtain a sequence of pairs $(u_k,w_k)=F_\A^k(u_0,w_0)$, $k\ge 0$, such that each geodesic $\g_k$ from $u_k$ to $w_k$ is $\A$-reduced and
\begin{equation}
    \label{wk} [w_k]_\A = \sequence{ n_k, n_{k+1}, \dots }.
\end{equation}
Using the bijectivity of the map $F_\A$, we extend the sequence (\ref{future}) to the past to obtain a bi-infinite sequence
\begin{equation}
    \label{codingseq} [\g]_\A = \sequence{ \dots,n_{-2},n_{-1},n_0,n_1,n_2,\dots },\quad n_i\in\Ab,
\end{equation}
called the \emph{arithmetic coding sequence} or \emph{arithmetic code} of $\g$ (or the \emph{coding sequence} or \emph{code} when the context is clear) as follows: from the bijectivity of the map $F_\A$ on $\Omega_\A$, there exists a pair $(u_{-1},w_{-1})\in\Omega_\A$ such that $F_\A(u_{-1},w_{-1})=(u_0,w_0)$, i.e., $f_\A w_{-1}=w_0$. Then $w_{-1}\in [A_i,A_{i+1})$ for some $i\in \Ab$, and for $n_{-1}=\sigma(i)$ we have
\[ [w_{-1}]_\A = \sequence{ n_{-1},n_0,n_1,\dots }. \]
Continuing inductively, we define the sequence $n_{-k}\in\Ab$ and the pairs $(u_{-k},w_{-k})\in\Omega_\A$ ($k\geq 2$), where
\[ [w_{-k}]_\A = \sequence{ n_{-k},n_{-k+1}, n_{-k+2},\dots }, \]
by $F_\A(u_{-k},w_{-k})=(u_{-(k-1)},w_{-(k-1)})$. We call the sequence
\begin{equation}
    \label{seq-past} \sequence[]{n_{-1}, n_{-2},\dots,n_{-k}\dots}
\end{equation}
the {\em past} of $\g$.

Notice that the future of $\g$ depends only on $w_0$ while, in general, the past of $\g$ depends on both, $w_0$ and $u_0$. In some rare cases the past only depends on $u_0$, and, in fact, the sequence (\ref{seq-past}) is an expansion of $u_0$ with respect to a different (dual) partition. This is a motivation for study dual codes in Section~\ref{dual codes}.

We also associate to $\g=\g_0$ a bi-infinite sequence $\{\g_k = u_kw_k \}_{k \in \Z}$ of $\A$-reduced geodesics $\Gamma$-equivalent to $\g$. The left shift of the this sequence corresponds to an application of the map $F_\A$ to the corresponding geodesic: $F_\A\g_k=\g_{k+1}$.

Combining results of Sections~\ref{sec conjugacy} and~\ref{sec cross-section}, we obtain the following result:
\begin{samepage}
\begin{prop}
    \label{coding} Let $\g$ be an $\A$-reduced geodesic on $\D$ and $\bar\g$ its projection to $M$. Then
    \begin{enumerate}
        \item each geodesic segment of $\bar\g$ between successive returns to the cross-section $C_\A$ produces an $\A$-reduced geodesic on $\D$, and each reduced geodesic $\Gamma$-equivalent to $\g$ is obtained this way;
		\item the first return of $\bar\g$ to the cross-section $C_\A$ corresponds to a left shift of the coding sequence of $\gamma$.
    \end{enumerate}
    Additionally, (1) and (2) hold for the case $\A = \P$.
\end{prop}
\end{samepage}

\begin{proof}
    Let $\g$ be an $\A$-reduced geodesic on $\D$. Then its projection $\bar\g$ to $M$ can be represented as a (countable) sequence of geodesic segments in $\F$. By Corollary~\ref{anti corner}, each such segment either extends to an $\A$-reduced geodesic on $\D$ or its image under $U_j$ is $\A$-reduced ($j$ is specified in Corollary~\ref{anti corner}). If $\g'$ is an $\A$-reduced geodesic $\Gamma$-equivalent to $\g$, then both project to the same geodesic in $M$. By Corollary~\ref{corner}, either $\g'$ intersects $\F$ or its image under $U_j^{-1}$ intersects $\F$ ($j$ is specified in Corollary~\ref{corner}). In either case the intersection of the corresponding geodesic on $\D$ with $\F$ is another segment of the same geodesic in $\F$. This completes the proof of (1).

Since $F_\A\g=\g_1$, $f_\A w_1=\sequence{n_1, n_2,\dots}$, the first digit of the past of $\g_1$ is $n_0$, and the remaining digits are the same as in the past of $\g$. Now (2) follows from Corollary~\ref{firstreturn}.

Since Corollaries~\ref{anti corner} and~\ref{corner} are true for $\A = \P$, the arguments in this proof hold for $\A = \P$ as well.
\end{proof}
The following corollary is immediate.
\begin{cor}
    \label{equivalentgeod} If $\g'$ is $\G$-equivalent to $\g$, and both geodesics can be reduced in finitely many steps, then the coding sequences of $\g$ and $\g'$ differ by a shift.
\end{cor}
Thus we can talk about coding sequences of geodesics on $M$. To any geodesic $\g$ that can be reduced in finitely many steps we associate the coding sequence (\ref{codingseq}) of a reduced geodesic $\G$-equivalent to it. Corollary~\ref{equivalentgeod} implies that this definition does not depend on the choice of a particular representative.

\medskip An \emph{admissible} sequence $s \in \Ab^\Z$ is one obtained by the coding procedure (\ref{codingseq}) from some reduced geodesic $uw$ with $(u,w)\in\Omega_\A$. Given an admissible sequence, then, we can associate to it the vector $v \in C_\A$ such that $v = \pi \circ \phi(u,w)$, where $\pi\circ\varphi:\Omega_{\A}\to C_{\A}$ is defined in (\ref{piphi}). We denote this vector $v$ by $\Cod(s)$. The map $\Cod$ is essentially bijective (finite-to one: see \renewcommand{\thesubsection}{\arabic{subsection}}Example~\ref{mult arith} in Section~\ref{sec examples}). By Proposition~\ref{prop:cont} below, $\Cod$ is uniformly continuous on the set of admissible coding sequences, and thus we can extend it to the closure $X_\A \subset \Ab^\Z$ of all admissible sequences.

The symbolic system $(X_{\A},\sigma)\subset (\Ab^\Z,\sigma)$ is defined on the alphabet $\Ab=\{1$, $2$, \ldots, $8g-4\}$. The product topology on $\Ab^\Z$ is induced by the distance function
$ d(s,s')=\frac 1{m}, $
where $m=\max\{\, k : s_i= s'_i\text{ for } \abs i \leq k \,\}$.
\begin{prop}
    \label{prop:cont} The map $\Cod$ is uniformly continuous.
\end{prop}

\begin{proof}
    Let $s$ and $s'$ be two admissible coding sequences obtained from reduced geodesics $uw$ and $u'w'$, respectively. If $d(s,s')<\frac 1{m}$, then the $\A$-expansions of the attracting end points $w$ and $w'$ of the corresponding geodesics given by (\ref{future}) have the same first $m$ symbols in their (forward) $\A$-expansions, and the same is true for any $y\in [w,w']$, hence $f^k_\A(y)=y_k\in[A_{\sigma(n_k)}, A_{\sigma(n_k)+1})$ for $0\leq k\leq m$, and $y=T_{n_0}T_{n_1}\cdots T_{n_m}(y_m)$.

Since $[A_i, A_{i+1})\subset (P_i, Q_{i+1})$, any point $y\in [A_i, A_{i+1})$ is inside the isometric circle for $T_i$, and thus $T'_i(y)>\mu_i>1$ for some $\mu_i > 1$. Using this for $i=\sigma(n_0), \sigma(n_1),\dots, \sigma(n_m)$ and the Chain Rule, we conclude that the derivative of the composite function
    \[ f_\A^{m+1}(y)=T_{\sigma(n_m)}\cdots T_{\sigma(n_1)}T_{\sigma(n_0)}(y) \]
    is $>\mu^{m+1}$, and hence the derivative of the inverse function
    \[ f_\A^{-(m+1)}(y_m)=T_{n_0}T_{n_1}\cdots T_{n_m}(y_m) \]
    is $<\lambda^{m+1}$ for some $\lambda<1$. Then, by the Mean Value Theorem, the arc length distance
    \[ \ell(w,w') = \big(f_\A^{-(m+1)}\big)'(c) \cdot \ell(f_\A^{-(m+1)}w, f_\A^{-(m+1)}w') \]
    for some $c \in [w,w']$. Since $(f_\A^{-(m+1)})'(c) < \lambda^{m+1}$ and the arclength distance between any two points is at most $2\pi$, we have that
    \[ \ell(w,w') \leq 2\pi \lambda^{m+1}. \]
    Now we choose $m$ large enough so that $w$ and $w'$ are so close that for any $x\in[u,u']$ $(x,y)\in \Omega_\A$ (here we use the rectangular structure of $\Omega_\A$ and are avoiding the corners). We recall that we used bijectivity of $F_\A$ to define the sequence of pairs $(x_{-i}, y_{-i})\in\Omega_\A$ to the past, $1\leq i\leq m$. Since the first $m$ symbols with negative indices in the coding sequences $s$ and $s'$ are the same, for any $y\in [w,w']$ we have $y_{-i}\in [A_{\sigma(n_{-i})}, A_{\sigma(n_{-i}+1)})$ and hence
    \[ f^m_\A(y_{-m})=T_{\sigma(n_{-1})}T_{\sigma(n_{-2})}\cdots T_{\sigma(n_{-m})}(y_{-m})=y, \]
    and at the same time we have
    \[ T_{\sigma(n_{-1})}T_{\sigma(n_{-2})}\cdots T_{\sigma(n_{-m})}(x_{-m})=x \]
    for $x\in[u,u']$.

But $T_{\sigma(n_{-i})}$ were determined by the second coordinate $y_{-i}$, and since $(x_{-i},y_{-i})\in\Omega_\A$, and $y_{-i}$ is inside the isometric circle for $T_{\sigma(n_{-i})}$, $x_{-i}$ is outside that isometric circle, so the derivative of the composite function $T_{\sigma(n_{-1})}T_{\sigma(n_{-2})}\cdots T_{\sigma(n_{-m})}(x_{-m})$ is $<\lambda^m$.
Therefore
    \begin{align*}
    	\ell(u, u') &= \ell\big(T_{\sigma(n_{-1})}T_{\sigma(n_{-2})}\cdots T_{\sigma(n_{-m})}(u_{-m}),T_{\sigma(n_{-1})}\cdots T_{\sigma(n_{-k})}(u'_{-m})\big) \\ &\leq 2\pi\lambda^{m}
	\end{align*}
    for some $\lambda<1$.

Therefore the geodesics are uniformly $2\pi\lambda^{m}$-close. But the tangent vectors $v = \Cod(s)$ and $v' = \Cod(s')$ in $C_{\A}$ are determined by the first intersection of the corresponding geodesic with the boundary of $\F$ or $\F_j$ for a particular $j$ determined in Corollary~\ref{anti corner}. Hence, by making $m$ large enough we can make $v'$ as close to $v$ as we wish. Thus we proved that the map $\Cod$ is uniformly continuous on the set of admissible coding sequences and therefore extends to the closure $X_\A$ as a uniformly continuous map.
\end{proof}

In conclusion, the geodesic flow becomes a special flow over a symbolic dynamical system $(X_{\A},\sigma)$ on the finite alphabet $\Ab$. The ceiling function $g_{\A}(s)$ on $X_{\A}$ is the time of the first return to the cross-section $C_\A$ of the geodesic associated to $s$.

\begin{figure}[htb]
\begin{tikzpicture}[scale=3.15]
    \fill [black!10] (0,0) circle (1);
    \draw [black!50] (0,0) node {$\F$};
    \begin{scope}
        \clip (0,0) circle (1);
        % white circles so F is colored in gray
        \foreach \k in {1,2,...,12}
            \fill [white] (-120+30*\k : 1.07457) circle (0.39332);
        % outlines of white circles from above
        \foreach \k in {1,2,...,12}
            \draw (-120+30*\k : 1.07457) circle (0.39332);
    \end{scope}
	\foreach \k/\s in {1/7,2/12,3/5,4/10,5/3,6/8,7/1,8/6,9/11,10/4,11/9,12/2} {
    	\draw (-120+\k*30:0.63) node [black] {\scriptsize$\k$};
		\draw (-120+\k*30:0.74) node [black] {\small$\boldsymbol{\s}$};
	}
    \draw (0,0) circle (1);
    \draw (0,-1.05) node [below] {(a)};
    
\begin{scope}[xshift=2.2cm]
    \begin{scope}
        \clip (0,0) circle (1);
        \foreach \k in {1,2,...,12} 
            \draw [rotate=-120+30*\k] (0:1.07457)+(-120:0.39332) arc (-120:120:0.39332);
    \end{scope}
	
	\foreach \k/\s in {1/7,2/12,3/5,4/10,5/3,6/8,7/1,8/6,9/11,10/4,11/9,12/2} {
		\draw [fill=black!10,rotate=150+30*\k] (-0.19666, 0.733945)--(-0.143549, 0.708381)--(-0.0872148, 0.691041)--(-0.0289217, 0.682315)--(0.0300209, 0.682397)--(0.0882893, 0.691287)--(0.144575, 0.708785)--(0.19666, 0.733945)--(0.173971, 0.779929)--(0.160673, 0.817184)--(0.152781, 0.847794)--(0.148146, 0.873406)--(0.145558, 0.895239)--(0.144308, 0.914184)--(0.143965, 0.930605)--(0.129178, 0.932429)--(0.116792, 0.937009)--(0.106917, 0.943184)--(0.099322, 0.950108)--(0.0936507, 0.957235)--(0.0895394, 0.964246)--(0.0867127, 0.970848)--(0.0800721, 0.969332)--(0.073605, 0.969404)--(0.0677094, 0.970805)--(0.0626219, 0.973204)--(0.0584347, 0.976273)--(0.0551365, 0.979732)--(0.0526949, 0.9833)--(0.0488263, 0.981624)--(0.0447405, 0.980878)--(0.0407093, 0.98105)--(0.0369651, 0.982038)--(0.033674, 0.983684)--(0.0309286, 0.985808)--(0.028793, 0.988191)--(0.0260533, 0.986569)--(0.0229979, 0.985571)--(0.0198088, 0.985263)--(0.0166785, 0.985645)--(0.0137804, 0.986651)--(0.0112466, 0.98817)--(0.00919302, 0.990029)--(0.00694077, 0.988362)--(0.00431751, 0.987165)--(0.00144915, 0.98654)--(-0.00150405, 0.986546)--(-0.00436923, 0.987182)--(-0.00698658, 0.988389)--(-0.00919302, 0.990029)--(-0.0112901, 0.988138)--(-0.0138316, 0.986627)--(-0.0167354, 0.985632)--(-0.0198684, 0.985262)--(-0.0230568, 0.985583)--(-0.0261079, 0.986594)--(-0.028793, 0.988191)--(-0.0309746, 0.985765)--(-0.0337307, 0.983648)--(-0.0370314, 0.982013)--(-0.0407826, 0.981039)--(-0.0448172, 0.980884)--(-0.0489014, 0.981647)--(-0.0526949, 0.9833)--(-0.0551903, 0.979666)--(-0.0585046, 0.976211)--(-0.0627088, 0.973152)--(-0.0678126, 0.970769)--(-0.0737216, 0.969389)--(-0.0801961, 0.969345)--(-0.0867127, 0.970848)--(-0.0896039, 0.964117)--(-0.0937413, 0.957102)--(-0.0994451, 0.949975)--(-0.107079, 0.94306)--(-0.117, 0.936906)--(-0.129433, 0.932366)--(-0.143965, 0.930605)--(-0.144322, 0.913852)--(-0.145593, 0.89486)--(-0.148211, 0.872966)--(-0.152894, 0.847273)--(-0.160864, 0.816555)--(-0.174292, 0.779157) -- cycle;
		\draw [black] (-120+\k*30:0.73) node {\scriptsize$\boldsymbol{\s}$};
		\draw [black!50] (-120+\k*30:0.87) node {\scriptsize$T_{\s}\F$};
	}
    \draw (0,0) circle (1);
    \draw (0,-1.05) node [below] {(b)};
\end{scope}
\end{tikzpicture}\vspace*{-1em}
\caption{(a) The ``outside'' labels in bold, corresponding to $T_k^{-1} = T_{\sigma(k)}$. (b)~These are the ``inside'' numbers for images of $\F$.}
\label{fig inside outside}
\end{figure}
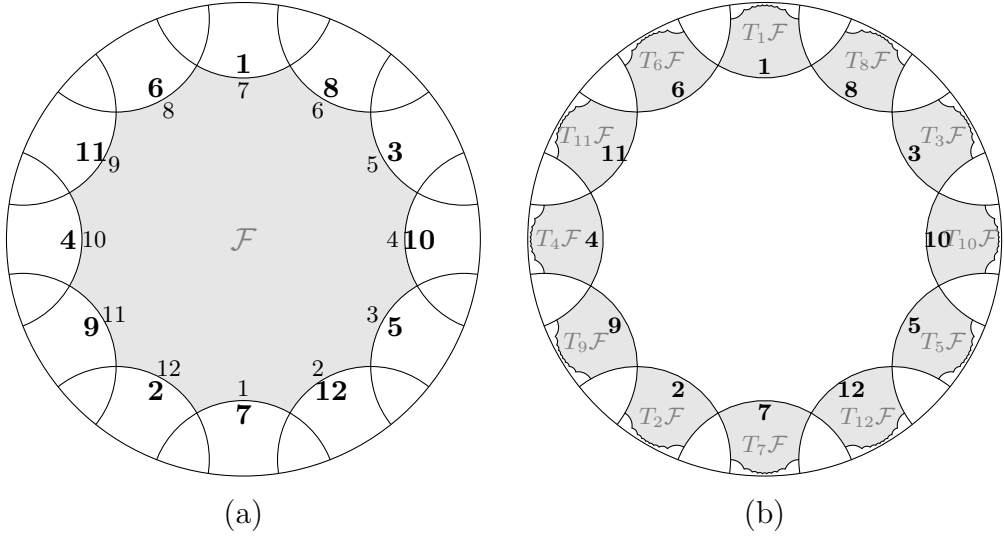

\section{Examples of coding} \label{sec examples}

Note that when a geodesic segment exits $\F$ through side $i$ it is $\sigma(i)$ that is used in the geometric code, as described in the Introduction. Similarly, if the endpoint of a geodesic lies in $[A_i,A_{i+1})$, it is $\sigma(i)$ that is used in the arithmetic code, as described in Section~\ref{sec coding}. We could call $i$ the ``inside'' numbering of the sides of $\F$ (this is the numbering used in the Introduction), while $\sigma(i)$ is the ``outside'' numbering, as shown in Figure~\ref{fig inside outside}. See \mbox{\cite[Introduction]{K96}} for more detail on geometric codes and numbering conventions.

%Given a geodesic $\gamma_0$ from $u_0$ to $w_0$, we define $(u_k,w_k) = F_\A^k(u_0,w_0)$ and then define $\gamma_k$ as the geodesic from $u_k$ to $w_k$. As described in Section~\ref{sec coding}, the arithmetic coding sequence \[ [\gamma]_\A = \sequence{\ldots,n_{-2},n_{-1},n_0,n_1,n_2,\ldots} \] uses $n_k = \sigma(i)$, where $w_k \in [A_i,A_{i+1})$.

\medskip
As described in \cite[Section~2]{KU07}, an axis of a transformation in $\G$ (that is, a geodesic whose beginning and end points are the repelling and attracting fixed points of the transformation) becomes a closed geodesic in $\G\backslash\D$ and has a periodic geometric code. For both arithmetic and geometric codes, we denote repeating codes by only showing the repetend, that is,
\[ \sequence[]{n_0,\ldots,n_k} \quad\text{means}\quad \sequence[]{\ldots,n_k,n_0,n_1,n_2,\ldots,n_k,n_0,\ldots}. \]

\begin{figure}
	\includegraphics[height=1.0\textheight]{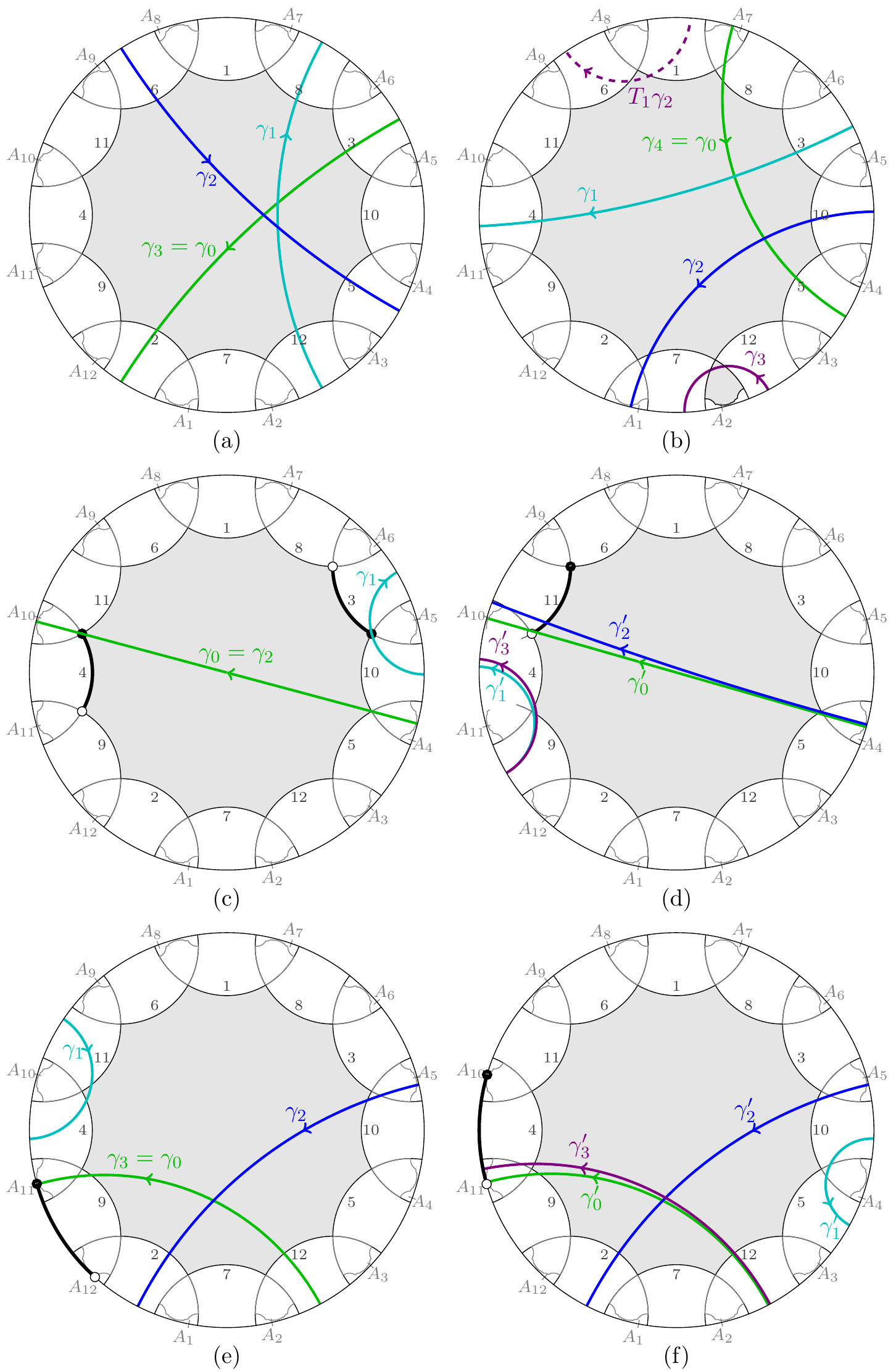}
	\caption{(a)~Example~1. (b)~Example~2. (c-d)~Ex.~3. (e-f)~Ex.~4.}
	\label{fig coding examples}
\end{figure}

\subsection{Example 1: codes agree.}

%The repelling and attracting fixed points of $T_2T_8T_5(z)$ are, respectively, \[ u_0 = e^{0.50589 \,i} \qquad\text{and}\qquad w_0 = e^{-2.13494 \,i}. \] For this example, we have
Let $\g_0$ be the axis of $T_2T_8T_5$. Then
\[
\begin{array}{r@{\;=\;}r@{\quad\text{because }}l}
    \gamma_1 & T_{12} \gamma_0  & w_0 \in [A_{12},A_1) \\
    \gamma_2 & T_6 T_{12} \gamma_0 & w_1 \in [A_6,A_7) \\
    \gamma_3 & T_3 T_6 T_{12} \gamma_0 & w_2 \in [A_3,A_4)
\end{array}
\] and $\gamma_3 = \gamma_0$. Thus the arithmetic code is
\[ [\gamma]_\A = \sequence{\sigma(12),\sigma(6),\sigma(3)} = \sequence{2,8,5}. \]
The geometric code for this example is the same because for each geodesic $\gamma_k$ the index~$i$ for which $w_k \in [A_i,A_{i+1})$ is the same $i$ as the side through which $\gamma_k$ exits $\F$; see Figure~\ref{fig coding examples}(a).
\subsection{Example 2: codes disagree.} 
%The repelling and attracting fixed points of $T_5T_4T_7T_6(z)$ are, respectively, \[ u_0 = e^{1.28092 \,i} \qquad\text{and}\qquad w_0 = e^{-0.54069 \,i}. \] For this example, we have
Let $\g_0$ be the axis of $T_5T_4T_7T_6$. Then
\[
\begin{array}{r@{\;=\;}r@{\quad\text{because }}l}
    \gamma_1 & T_3 \gamma_0  & w_0 \in [A_3,A_4) \\
    \gamma_2 & T_{10} T_3 \gamma_0 & w_1 \in [A_{10},A_{11}) \\
    \gamma_3 & T_{12} T_{10} T_3 \gamma_0 & w_2 \in [A_{12},A_1) \\
    \gamma_4 & T_1 T_{12} T_{10} T_3 \gamma_0 & w_3 \in [A_1,A_2)
\end{array}
\] and $\gamma_4 = \gamma_0$. The arithmetic code is thus
\[ [\gamma]_\A = \sequence{\sigma(3),\sigma(10),\sigma(12),\sigma(1)} = \sequence{4,5,2,7}. \]
To find the geometric code of $\gamma$, we look at the sides through which segments exit $\F$. $\gamma_0$ exits through side $3$ (``inside'' numbering), so $T_3 \gamma = \gamma_1$ is the next segment. Then $\gamma_1$ exits through side $10$ and $T_{10} T_3 \gamma = \gamma_2$ is the next segment. At this point, we see that $\gamma_2$ exits $\F$ through side $1$, so the geometrically next segment is $T_1 T_{10} T_3 \gamma = T_1 \gamma_2$ (this is \textit{not}~$\gamma_3$). The geodesic $T_1 \gamma_2$ is shown as a dashed curve at the top of Figure~\ref{fig coding examples}(b). It exits $\F$ through side $8$, and then $T_6(T_1 \gamma_2) = \gamma_0$, so the geometric code for this example is
\[ [\gamma]_G = \sequence[G]{\sigma(3),\sigma(10),\sigma(1),\sigma(8)} = \sequence[G]{5,4,7,6}. \]
The geometric code $\sequence[G]{5,4,7,6}$ differs from the arithmetic code $\sequence{4,5,2,7}$ for this example precisely because $\gamma_2$ exits through side $1 = \sigma(7)$ but ends at $w_2 \in [A_{12},A_1)$, not in the interval $[A_1,A_2)$. Note that $\g_3$ intersects the corner image $\F_2$ and that $T_1 \gamma_2$ is exactly $U_2^{-1} \gamma_3$.

\subsection{Example 3: multiple geometric codes.}\label{mult geom} 

The axis of $U_{10}$ goes from $M_4$ to $M_{10}$ and passes directly through the vertex $V_{10}$ as it exits $\F$.

By convention, each vertex $V_i$ is considered to be part of side $i$. Using this convention, the geometric code of $\g$ should start with $\sigma(10) = 4$ (see Figure~\ref{fig coding examples}(c), in which side $10$ is shown with a thick black arc). Then $\g_1 = T_{10} \g_0$ is seen to pass through vertex $V_5$, so we code with $\sigma(5) = 3$ (side $5$ also has a thick black arc in Figure~\ref{fig coding examples}(c)). The next geodesic is $\g_2 = T_5 \g_1 = T_5 T_{10} \g_0 = U_4 \g_0$, which is exactly $\g_0$ because $U_4$ fixes $M_4$ and $M_{10}$ by Lemma~\ref{U fixes M}. The conventional geometric code of $\g$ is therefore $\sequence[G]{4, 3}$.

Now consider a geodesic $\g'$ very close to $\g = M_4M_{10}$ that exits $\F$ through side $9$ instead of side $10$, see Figure~\ref{fig coding examples}(d). Its geometric code starts with $\sigma(9) = 11$, and then $\g_1' = T_9 \g'$ will exit through side $10$, so the code continues with $\sigma(10) = 4$. The geodesic $\g_2' = T_{10} T_9 \g'$ will be somewhat close to $\g$, but later iterates $\g_{2k}'$ will eventually stop exiting through side~$9$. However, taking $\g'$ sufficiently close to $\g$ will give geodesics whose forward geometric codes begin with arbitrarily many repetitions of $11$ and $4$. For example, there might be a geodesic with code 
\[ \sequence[G]{\ldots,5,7,\;\;11,4,11,4,11,4,11,4,\;\;8,3,1,\ldots}. \]
The repeating code $\sequence[G]{11,4}$ will not be the admissible geometric code of this geodesic (or in fact any geodesic), but it is in the closure of the space of admissible geometric codes. Note that the transformation $T_{11}T_4$ is exactly equal to $T_4 T_3$ by \eqref{eq:Ui}, as both represent $U_{10}$.

The codes $\sequence[G]{4,3}$ and $\sequence[G]{11,4}$ both code the geodesic $M_4M_{10}$. Using one of $\sequence[G]{4,3}$ or $\sequence[G]{11,4}$ in the future and the other in the past gives two non-periodic codes as well.

\subsection{Example 4: multiple arithmetic codes.}\label{mult arith}

The repelling and attracting fixed points of $T_4T_5T_2(z)$ are, respectively, $e^{-1.07822 \,i}$ and $e^{-2.86313 \,i}$.
 The repelling point $w_0$ is in $(P_{11},Q_{11})$, and for our particular choice of $\A$ in these examples it is exactly~$A_{11}$.

In Figure~\ref{fig coding examples}(e), the interval $[A_{11},A_{12})$ is shown with a thick black arc. Using the convention that $T_i$ be applied to $A_i$, we have
\[
\begin{array}{r@{\;=\;}r@{\quad\text{because }}l}
    \gamma_1 & T_{11} \gamma_0  & w_0 \in [A_{11},A_{12}) \\
    \gamma_2 & T_{10} T_{11} \gamma_0 & w_1 \in [A_{10},A_{11}) \\
    \gamma_3 & T_{12} T_{10} T_{11} \gamma_0 & w_2 \in [A_{12},A_1)
\end{array}
\] and $\gamma_3 = \gamma_0$.  The arithmetic code is thus
\[ [\gamma]_\A = \sequence{\sigma(11),\sigma(10),\sigma(12)} = \sequence{9,4,2}. \]

Now consider a geodesic $\g'$ very close to $\g$ but with attracting endpoint slightly clockwise of $A_{11}$, see Figure~\ref{fig coding examples}(f). Let us denote the sequence of its iterates under $F_\A$ by $\g_1'$, $\g_2'$, etc.
\[
\begin{array}{r@{\;=\;}r@{\quad\text{because }}l}
    \gamma_1' & T_{10} \gamma_0'  & w_0' \in [A_{10},A_{11}) \\
    \gamma_2' & T_{3} T_{10} \gamma_0' & w_1' \in [A_{3},A_{4}) \\
    \gamma_3' & T_{12} T_{3} T_{10} \gamma_0' & w_2' \in [A_{12},A_1)
\end{array}
\]
The geodesic $\g_3'$ will be somewhat close to $\g$, but later iterates $\g_{3k}'$ will eventually not end in $[A_{10},A_{11})$. However, taking $\g'$ sufficiently close to $\g$ will give geodesics whose forward arithmetic codes have arbitrarily many repetitions of \[ \sequence[]{\sigma(10),\sigma(3),\sigma(12)} = \sequence[]{4,5,2}. \]
The repeating code $\sequence{4,5,2}$ will not be the admissible arithmetic code of this geodesic (or in fact any geodesic), but it is in the closure $X_\A$ of the space of admissible arithmetic codes.

The codes $\sequence{9,4,2}$ and $\sequence{4,5,2}$ both code the geodesic $\g$. Using one of $\sequence{9,4,2}$ or $\sequence{4,5,2}$ in the future and the other in the past gives two non-periodic codes as well.

It is worth noting that the transformations $T_4T_5T_2$ and $T_9T_4T_2$ are identical: restating~\eqref{r13} as $T_{\rho(i)} T_i = (T_{\rho^2(i)})^{-1} (T_{\rho^3(i)})^{-1} = T_{\sigma\rho^2(i)} T_{\sigma\rho^3(i)}$, we have that
\[ (T_4 T_5)T_2 = (T_{\rho(5)} T_5)T_2 = (T_{\sigma\rho^2(5)} T_{\sigma\rho^3(5)})T_2 = (T_{\sigma(11)} T_{\sigma(10)})T_2 = (T_9 T_4)T_2. \]

\section{Parameters admitting Markov/sofic partitions}\label{sec Markov}

Adler and Flatto~\cite{AF91} show that the symbolic system associated to the map we call $F_\P : \Omega_\P \to \Omega_\P$ is sofic with respect to the alphabet $\Ab = \{1,2,\ldots,8g-4\}$. This is not always possible for $F_\A : \Omega_\A \to \Omega_\A$, but there are several examples of $\A$ for which the shift is sofic. In this section and the next, we first give a sufficient condition for sofic shifts to occur, and then we give some examples.

As described in~\cite[Appendix C]{AF91}, sofic systems are obtained from Markov ones by amalgamation of the alphabet. Conversely, Markov is obtained from sofic by refinement of the alphabet, which in this case is realized by splitting each symbol $i \in \Ab$ into three symbols $i_1, i_2, i_3$. A partition of $\Omega_\A$ whose shift is Markov with respect to the \emph{extended alphabet} $\{\, i_k : i \in \Ab, k=1,2,3 \,\}$ will be sofic with respect to the alphabet $\Ab$.
\begin{remark}
    The value $U^{-1}_i A_i$ is a ``cycle end'' (see~\cite[Sec.~3]{KU17}) because
    \[ T_k B_k = T_k T_{\sigma(k-1)} A_{\sigma(k-1)} = U^{-1}_{\sigma(k-1)} A_{\sigma(k-1)}. \]
\end{remark}

\begin{prop}
    \label{Markov if} Let $\A$ have the short cycle property. A Markov partition (with respect to the extended alphabet) exists for $F_\A : \Omega_\A \to \Omega_\A$ if for all~$i \in \A$ there exists a~$j \in \A$ such that $U_i^{-1}A_i \in \{ A_j, B_j, C_j \}$.
\end{prop}

\begin{proof}
    Following the notion of the ``fine partition'' in~\cite{AF91}, we split each strip of $\Omega_\A $ with $w \in [A_i,A_{i+1})$ into three rectangles $R_{i_1}$, $R_{i_2}$, and $R_{i_3}$. Assuming that $C_i \in (A_i,B_i]$, define
    \begin{align*}
        R_{i_1} &:= [Q_{i+1},P_{i-1}] \times [A_i,C_i] \\
        R_{i_2} &:= [Q_{i+2},P_{i-1}] \times [C_i,B_i] \\
        R_{i_3} &:= [Q_{i+2},P_i] \times [B_i,A_{i+1}]
    \end{align*}
    (if $B_i \in (A_i,C_i)$ instead, the process is very similar except that $[A_i,A_{i+1}]$ is partitioned into $[A_i,B_i]$, $[B_i,C_i]$, and $[C_i,A_{i+1}]$). The left side of Figure~\ref{fig fine} shows the partition \mbox{$\{\, R_{i_k} : i \in \Ab, k=1,2,3 \,\}$} when each $A_i$ is the midpoint of $[P_i,Q_i]$. Note that each $R_{i_2}$ is extremely thin.

In accordance with~\cite[Theorem 7.9]{A98}, it is sufficient to prove that for any pair of distinct symbols $i_k$ and $j_\ell$, the rectangles $F_\A(R_{i_k})$ and $R_{j_\ell}$ are either disjoint (that is, their interiors are disjoint) or intersect ``transversally,'' i.e., their intersection is a rectangle with two horizontal sides belonging to the horizontal boundary of $R_{j_\ell}$ and two vertical sides belonging to the vertical boundary of $F_\A(R_{i_k})$. The images $F_\A(R_{i_k})$ are
    \begin{align*}
        F_\A(R_{i_1}) &= [T_i Q_{i+1},T_i P_{i-1}] \times [T_i A_i,T_i C_i] \\
        &= [P_{\sigma(i)}, P_{\sigma(i)+1}] \times [B_{\sigma(i)+1}, U^{-1}_{\tau(\sigma(i))}A_{\tau(\sigma(i))}] \\
        F_\A(R_{i_2}) &= [T_i Q_{i+2},T_i P_{i-1}] \times [T_i C_i,T_i B_i] \\
        &= [Q_{\sigma(i)}, P_{\sigma(i)+1}] \times [U^{-1}_{\tau(\sigma(i))}A_{\tau(\sigma(i))}, U^{-1}_{\tau(\sigma(i))+1}A_{\tau(\sigma(i))+1}] \\
        F_\A(R_{i_3}) &= [T_i Q_{i+2},T_i P_i] \times [T_i B_i,T_i M_{i+1}] \\
        &= [Q_{\sigma(i)}, Q_{\sigma(i)+1}] \times [U^{-1}_{\tau(\sigma(i))+1}A_{\tau(\sigma(i))+1}, C_{\sigma(i)-1}].
    \end{align*}
    From this, it can be seen that $F_\A(R_{i_k}) \cap R_{j_\ell}$, if it is non-empty, will be a rectangle with vertical sides from $\P \cup \Q$ as required since the rectangles $R_{i_1},R_{i_2},R_{i_3}$ also use $\P \cup \Q$ for the values of their vertical sides. Let $m = \tau(\sigma(i))$. The horizontal sides may be $B_{\sigma(i)+1}$ or $C_{\sigma(i)-1}$ or may be one of the values $U^{-1}_{m}A_{m}$ or $U^{-1}_{m+1}A_{m+1}$, which in general are not values of the horizontal sides of the original partition elements. Thus the partition \mbox{$\{\, R_{i_k} : i \in \Ab, k=1,2,3 \,\}$} is Markov exactly when the values $U^{-1}_m A_m$ and $U^{-1}_{m+1}A_{m+1}$ are already horizontal sides of $R_{i_1},R_{i_2},R_{i_3}$. This happens exactly when $U^{-1}_m A_m$ and $U^{-1}_{m+1} A_{m+1}$ are equal to some $A_j$, $B_j$, or $C_j$.
\end{proof}

\begin{figure}[ht]
    \begin{center}
        \includegraphics[width=0.4\textwidth]{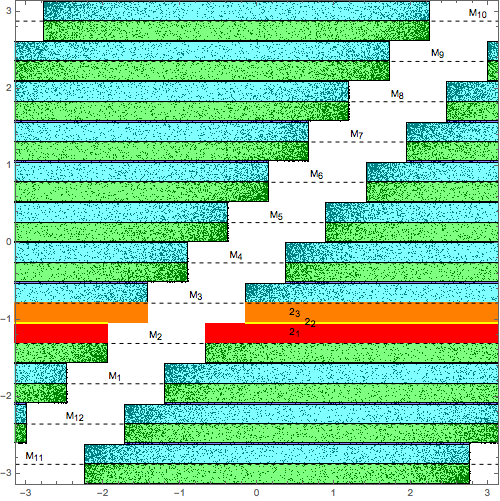} \;\;\raisebox{7em}{$\stackrel{F_{\! \bar M}\,}{\longrightarrow}$}\quad \includegraphics[width=0.4\textwidth]{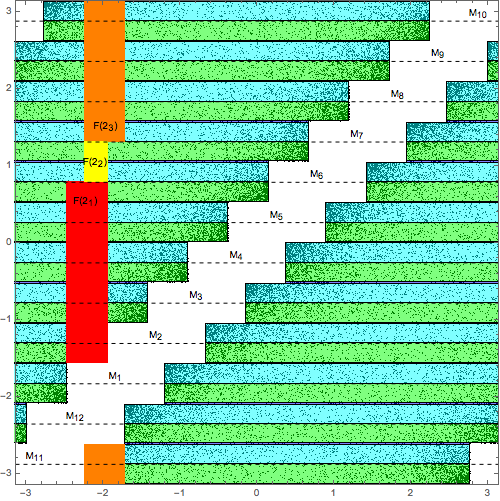}
    \end{center}\vspace*{-0.5em}
    \caption{Markov partition of $\Omega_{\bar M}$ for genus $g=2$.} \label{fig fine}
\end{figure}

The most notable example of a partition $\A$ satisfying the condition of Proposition~\ref{Markov if} is the ``midpoint'' setup, in which each $A_i$ is exactly halfway between $P_i$ and $Q_i$. We label these midpoints $M_i$ and write $F_{\bar M}$ for this specific case of $F_\A$. Note that each $M_i$ does have the short cycle property, so all the previous results on short cycles hold for $\A = \bar M$.

To show that $\A = \bar M$ satisfies the condition of Proposition~\ref{Markov if}, we need that $U^{-1}_i M_i$ is equal to some $M_j$, $B_j$, or $C_j$. In fact, $U^{-1}_i M_i = M_i$ by the following lemma.
\begin{lem}
    \label{U fixes M} The fixed points of $U_i$ are $M_i$ and $M_{\tau(i)}$.
\end{lem}

\begin{proof}
    Consider the image under $U_i = T_{\sigma(i)} T_{\tau(i)-1}$ of the vertices $V_i$ and $V_{\tau(i)}$. To shorten notation, let $C_j$ be the isometric circle $P_jQ_{j+1}$, which is the extension of side $j$ of $\F$. Then each vertex $V_j$ is the intersection of $C_j$ and $C_{j-1}$. The map $T_{\tau(i)-1}$ sends $V_{\tau(i)}$ to $V_{\sigma(i)+1}$ because it maps the isometric circles $C_{\tau(i)-1}$ and $C_{\tau(i)}$ to $C_{\sigma(i)+1}$ and $C_{\sigma(i)}$, respectively. Then $T_{\sigma(i)}$ sends $V_{\sigma(i)+1}$ to $V_i$ because it maps $C_{\sigma(i)+1}$ and $C_{\sigma(i)}$ to $C_{i-1}$ and $C_i$, respectively. Thus $U_i(V_{\tau(i)}) = V_i$.

The geodesic connecting $V_i$ and $V_{\tau(i)}$ in the disk model is a Euclidean line segment through the origin. Call this segment $L$. The isometric circles $C_{\tau(i)-1}$ and $C_{\tau(i)}$ intersect perpendicularly at $V_{\tau(i)}$, and $L$ bisects the right angle $V_{\tau(i)-1}V_{\tau(i)}V_{\tau(i)+1}$. Thus the image $U_i L$ is a geodesic in $U_i \F$ that bisects the right angle $(U_i V_{\tau(i)-1})V_i(U_i V_{\tau(i)+1})$. This means that $U_i L$ is the extension of the line $L$ into $U_i\F$. 

The full geodesic containing $L$ is precisely the line connecting $M_i$ and $M_{\tau(i)}$, and this line must (as a set) be fixed by $U_i:\D \to \D$. Since $U_i$ maps $\partial \D = \mathbb S$ to itself, the fixed points of $U_i$ are precisely the two points $M_i$ and $M_{\tau(i)}$.
\end{proof}

\section{Examples of Markov/sofic setups} \label{sec examples of Markov}

We now give a few examples of choices for~$\A$ that satisfy the condition of Proposition~\ref{Markov if} and therefore have a symbolic system that is sofic with respect to the alphabet $\Ab = \{1,2,\ldots,8g-4\}$.

\subsection*{Example 1: midpoints.} Suppose $U^{-1}_i A_i = A_i$ for all $i$. Then by Lemma~\ref{U fixes M}, $A_i = M_i$ (since $M_{\tau(i)}$ is not in $(P_i,Q_i)$, it cannot be $A_i$). See Figure~\ref{fig fine} for the actual attractor and Markov partition when $g=2$.

\subsection*{Example 2.} Suppose $U^{-1}_i A_i = A_{i+1}$ for all $i$. Then we can build a sequence of equations
\begin{align*}
    A_{2} &= U_1^{-1} A_1 \\
    A_{3} &= U_2^{-1} A_2 = U_2^{-1} U_1^{-1} A_1 \\
    A_{4} &= U_3^{-1} A_3 = U_3^{-1} U_2^{-1} U_1^{-1} A_1 \\
    &\;\;\vdots \\
    A_1 &= U_{8g-4}^{-1} A_{8g-4} = U_{8g-4}^{-1} U_{8g-5}^{-1} \cdots U_2^{-1} U_1^{-1} A_1
\end{align*}
giving that $A_1$ has to be a fixed point of $U_1 U_2 \cdots U_{8g-4}$. By Lemma~\ref{product fp} below, this product does have an attracting fixed point in $[b_1,a_1] \subset (P_1,Q_1)$, so we can choose $A_1$ to be this fixed point. The remaining $A_i$, $i\ne1$, will be the attracting fixed points of the cyclic product $U_i U_{i+1} \cdots U_{i-2} U_{i-1}$, where the indices are mod $8g-4$. By the cyclic permutation of indices in Lemma~\ref{product fp}, $A_i\in[b_i,a_i]$. Notice that $A_i$, $i\ne1$, will be shifts of $A_1$ by $(i-1)\smash{\raisebox{0.2em}{$\frac{\pi}{4g-2}$}}$ since each $U_i$ differs from $U_1$ by conjugation with a rotation.

\begin{lem}
    \label{product fp} The product $U_1 U_2 \cdots U_{8g-4}$
    has an attracting fixed point in $[b_1,a_1]$.
\end{lem}

\begin{proof}
    Recall that $a_i=T_{\sigma(i)}P_{\rho(i)+1}$ and $b_i=T_{\sigma(i-1)}Q_{\theta(i-1)}$, and any $A_i\in[b_i,a_i]$ satisfies the short cycle property (see~\cite[Corollary 8.2(i)]{KU17}). Using Proposition~\ref{tau properties} we obtain
    \[
    \begin{aligned}
        a_i&=T_{\sigma(i)}P_{\rho(i)+1}=T_{\sigma(i)}P_{\sigma(\tau(i-1))+1}= T_{\sigma(i)}T_{\tau(i-1)}P_{\tau(i)-2}=U_iP_{\tau(i)-2}\\
        b_i&=T_{\sigma(i-1)}Q_{\theta(i-1)}=T_{\sigma(i-1)}Q_{\sigma(\tau(i))}= T_{\sigma(i-1)}T_{\tau(i)}Q_{\tau(i)+2}=U_iQ_{\tau(i)+2}
    \end{aligned}
    \]
    Then $U_i([Q_{\tau(i)+2},P_{\tau(i)-2}])\subset[b_i,a_i]$. Take $x\in[b_1,a_1]$ and look at its orbit under a descending product of maps $U_i$:
    \begin{align*}
        U_{8g-4}(x) &\in [b_{8g-4},a_{8g-4}] \\
        U_{8g-5}U_{8g-4}(x) &\in [b_{8g-5},a_{8g-5}] \\
        U_{8g-6}U_{8g-5}U_{8g-4}(x) &\in [b_{8g-6},a_{8g-6}] \\
        &\;\;\vdots \\
        U_1 U_2 \cdots U_{8g-4}(x) &\in [b_{1},a_{1}].
    \end{align*}
    Continuing, we obtain a sequence of points $x_n=(U_1 U_2 \cdots U_{8g-4})^n(x)\in[b_1,a_1]$, which by compactness has a limit point in the set, and this point is the attracting fixed point since the attracting fixed point is unique.
\end{proof}
Examples 1 and 2 are ``equally spaced'' in the sense that the arclength $\ell(A_i,A_{i+1})$ is the same for each $i$. The following example does not have this property. 

\subsection*{Example 3.} Suppose $U^{-1}_i A_i = A_i$ for odd $i$ and $U^{-1}_i A_i = A_{i+1}$ for even $i$. Then by Lemma~\ref{U fixes M}, $A_i = M_i$ for odd $i$. For even $i$, we get that $A_i = U_i M_{i+1}$ simply by applying $U_i$ to both sides of $U_i^{-1} A_i = M_{i+1}$. Since $M_{i+1}$ is outside of the isometric circle for $U_i$, we see that indeed $A_i=U_i M_{i+1} \in (P_i,Q_i)$.

\section{Dual codes} \label{dual codes}

In~\cite[Sec.~5]{KU12}, Katok and Ugarcovici discuss cases in which expansions using two different parameters are ``dual''. The corresponding definition in current setup would be the following:
 
\begin{defn}
    Let $\A=\{A_i\}$ and $\A'=\{A_i'\}$ be two partitions (not necessary having the short cycle property) such that $F_\A$ and $F_{\A'}$ have attractors $\Omega_\A$ and $\Omega_{\A'}$, respectively, with finite rectangular structure, and let $\phi(x, y) = (y, x)$ be the reflection of the plane about the line $y = x$. We say that $\A$ and $\A'$ are \emph{dual} (equivalently, each is the \emph{dual} of the other) if $\Omega_{\A'}=\phi(\Omega_\A)$ and the following diagram is commutative:
    \begin{equation}
        \label{cd:jux} \commutativeDiagram{\Omega_\A}{\phi}{\Omega_{\A'}}{F^{-1}_\A}{F_{\A'}}{\Omega_\A}{\phi}{\Omega_{\A'}}.
    \end{equation}
\end{defn}
The motivation for studying dual codes is given in the next result. 
In order to state it, we must first introduce the \emph{backward $A'$-expansion} of $u_0$, defined as
\begin{equation}
    \label{past} [u_0]_{\A'}^- := \sequence[\A']{ m_{-1}, m_{-2}, m_{-3},\dots },
\end{equation}
where, for $k<0$, $m_k=i\in\Ab$ if $f_{\A'}^{-k+1}(u_0)\in [A'_i,A'_{i+1})$.
\begin{thm}
    Suppose $\A$ admits a dual $\A'$, both $F_\A$ and $F_{\A'}$ have attractors $\Omega_\A$ and $\Omega_{\A'}$, respectively, with finite rectangular structure, and let $\g$ be an $\A$-reduced geodesic from $u$ to $w$. Then the coding sequence
    \begin{equation}
        \label{codingseq1} [\g]_\A =\sequence{ \dots,n_{-2},n_{-1},n_0,n_1,n_2,\dots },\quad n_i\in\Ab,
    \end{equation}
    is obtained by juxtaposing the forward $\A$-expansion of $w$, $[w]_\A^+$, and the backward $\A'$-expansion of $u$, $[u]_{\A'}^-$. This property is preserved under the left shift of the sequence.
\end{thm}

\begin{proof}
    Let $(u,w)\in \Omega_\A$. Then $\phi(u,w)=(w,u)\in \Omega_{\A'}$. Using (\ref{cd:jux}) we obtain
    \[ F_{\A'}(w,u)=\phi\circ F^{-1}_{\A}(u,w)=(w_{-1},u_{-1}). \]
    Let $w_{-1}\in (A_i,A_{i+1})$. Then, by definition, the forward $\A$-expansion of $w_{-1}$ begins with $n_{-1}=\sigma(i)$, $f_\A =T_i$, and
    \[ F_\A(u_{-1}, w_{-1}) = (T_iu_{-1}, T_iw_{-1}) = (u,w). \]
    Therefore $w_{-1}=T^{-1}_iw=T_{\sigma(i)}w$.

	On the other hand, let
    \[ u = \sequence[\A']{m_{-1},m_{-2},\dots }. \]
    This means that $u\in [A'_{m_{-1}}, A'_{m_{-1}+1})$ and
    \[ F_{\A'}(w,u)=T_{m_{-1}}(w,u)=(T_{m_{-1}}w, T_{m_{-1}}u)=(w_{-1},u_{-1}). \]
    Therefore $T_{m_{-1}}=T_{\sigma(i)}$ which implies $m_{-1}=n_{-1}$. Continuing by induction, one proves that all digits of the ``past'' of the sequence (\ref{codingseq1}) are the digits of the backward $\A'$-expansion of $u$. 

	In order to see what happens under a left shift, we reverse the diagram to obtain
    \begin{equation}
        \label{cd:jux1} \commutativeDiagram{\Omega_\A}{\phi}{\Omega_{\A'}}{F_\A}{F^{-1}_{\A'}}{\Omega_\A}{\phi}{\Omega_{\A'}}.
    \end{equation}
    Let $(u,w)\in\Omega_\A$. Using (\ref{cd:jux1}) we obtain
    \[ F^{-1}_{\A'}(w,u)=\phi\circ F_{\A}(u,w)=(w_{1},u_{1}). \]
    Let $w\in (A_i,A_{i+1})$. Then, by definition, the forward $\A$-expansion of $w$ begins with \mbox{$n_{0}=\sigma(i)$}, $f_\A =T_i$, and
    \[ F_\A(u, w) = (T_iu, T_iw) = (u_1,w_1). \]
    Therefore $w_{1}=T_iw$.

	On the other hand, let
    \[ u_1 = \sequence[\A']{ m_{0},m_{-1},m_{-2},\dots }. \]
   	\begin{samepage}
    This means that $u_1\in [A'_{m_{0}}, A'_{m_{0}+1})$ and
    \[ F_{\A'}(w_1,u_1)=T_{m_{0}}(w_1,u_1)=(T_{m_{0}}w_1, T_{m_{0}}u_1)=(w,u). \]
    Therefore $w_1=T^{-1}_{m_0}w=T_{\sigma(m_0)}w$. Thus $T_{\sigma(m_{0})}=T_{i}$ which implies $m_{0}=n_{0}=\sigma(i)$.
	\end{samepage}
\end{proof}

\begin{figure}[ht]
    \begin{multicols}{2}
    		\includegraphics{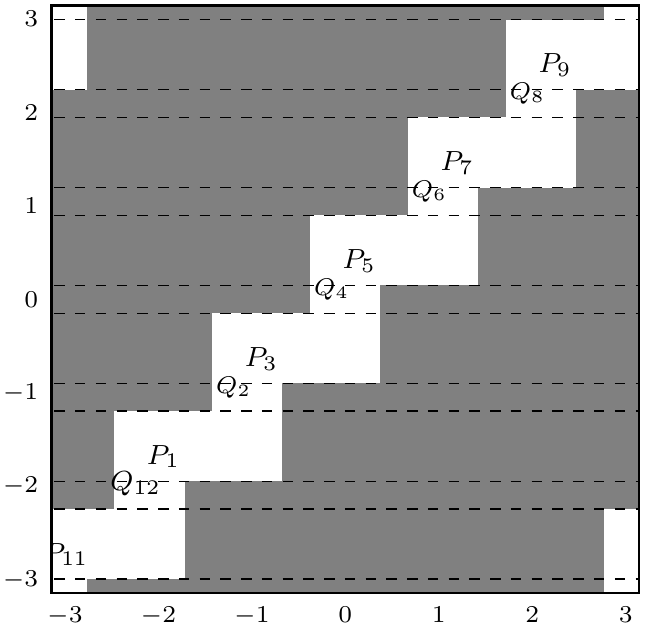} \\
    		{(a) Attractor $\Omega_{\,\overline{\!PQ\!}\,}$.} \\[1em]
    	
    		\includegraphics{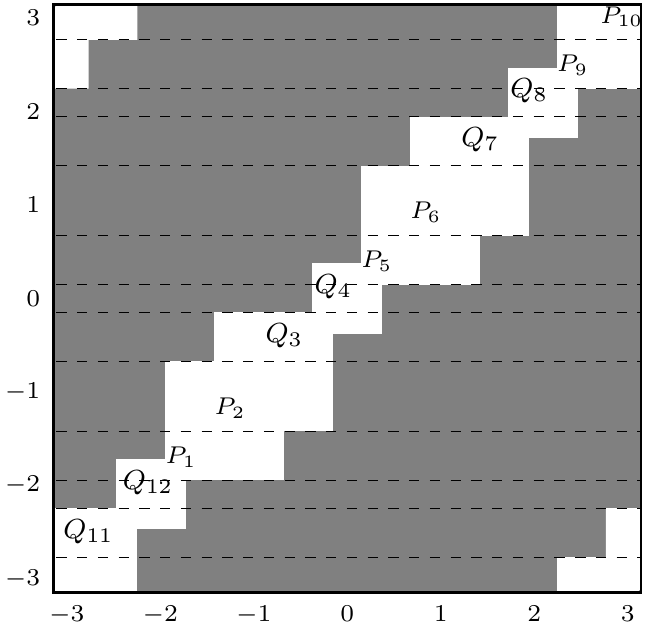} \\
    		{(c) Attractor $\Omega_{\,\overline{\!PPQQ\!}\,}$.}
		
		\columnbreak
		
			\includegraphics{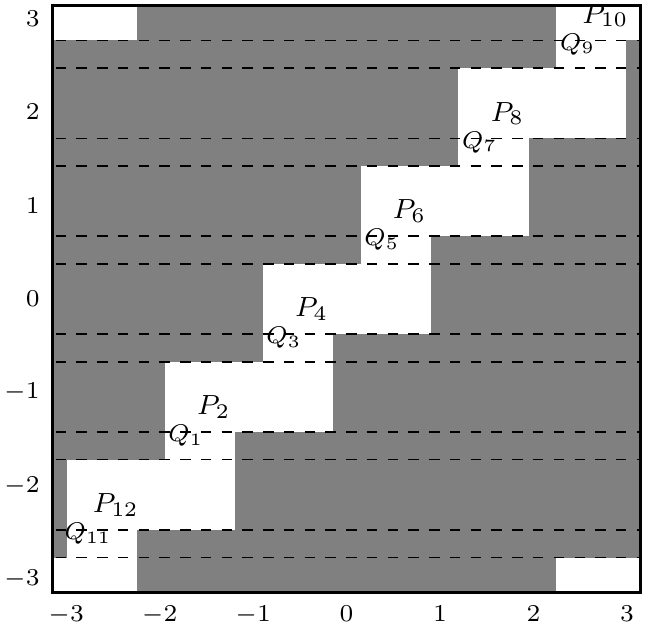} \\
    		{(b) Attractor $\Omega_{\,\overline{\!QP\!}\,}$.} \\[1em]
		
    		\includegraphics{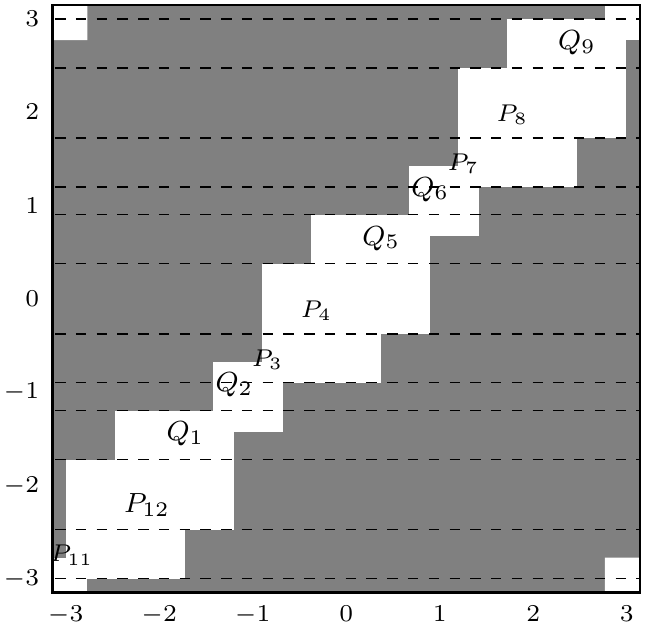} \\
    		{(d) Attractor $\Omega_{\,\overline{\!QQPP\!}\,}$.}
    \end{multicols}
	\caption{Attractors for different partitions, all with $g=2$.}
	\label{fig endpoints}
\end{figure}

By inspection, one can see that the partitions $\P$ and $\Q$ are dual. Future work will investigate other instances of duality among maps for which each $A_i$ is either $P_i$ or $Q_i$. Without giving the proofs here, it is interesting to note that the partitions
\begin{align*}
    {\,\overline{\!PQ\!}\,} &= \{ P_1, Q_2, P_3, Q_3, \ldots, P_{8g-5}, Q_{8g-4} \} \\
    {\,\overline{\!QP\!}\,} &= \{ Q_1, P_2, Q_3, P_4, \ldots, Q_{8g-5}, P_{8g-4} \}
\end{align*}
are dual to each other and that each of
\[ {\,\overline{\!PPQQ\!}\,} = \{ P_1, P_2, Q_3, Q_4, P_5, P_6, \ldots \} \]
and
\[ {\,\overline{\!QQPP\!}\,} = \{ Q_1, Q_2, P_3, P_4, Q_5, Q_6, \ldots \} \]
is self-dual ($\A' = \A$). Proofs of these dualities, and indeed of the structure of attractors for these specific partitions are known at this time (see Figure~\ref{fig endpoints}), but the general case of $A_i \in \{P_i,Q_i\}$ is not completely understood.

Restricting to short cycles with $A_i \in (P_i,Q_i)$ in fact yields no examples of duality:
\begin{prop}
    \label{no dual} There do not exist $\A$ and $\A'$, where $A_i,A_i' \in (P_i,Q_i)$, such that $\A$ and $\A'$ are dual and both satisfy the short cycle property.
\end{prop}

\begin{proof}
    Recall that for short cycles the upper part of $\Omega_\A $ has corner points
    \[ (P_i,B_i) \quad\text{(upper part)} \qquad\text{and}\qquad (Q_{i+1},C_i) \quad\text{(lower part)}. \]
    where $B_i$ and $C_i$ are defined in \eqref{B and C}; we define $B_i'$ and $C_i'$ similarly based on $\A'$.

Suppose we have $\A$ and $\A'$ dual with short cycles. Then $\psi(P_i,B_i) = (B_i,P_i)$ is a lower corner point for $\Omega_{\A'}$, meaning that $(B_i,P_i) = (Q_{j+1},C_j')$ for some $j$. Specifically, $B_i = Q_{j+1}$. But since $B_i \in (Q_i,A_{i+1})$ by~\cite[Def.~3.10]{KU17}, $B_i$ cannot be $Q_{j+1}$ for any $j$, which leads to a contradiction.

This contradiction technically completes the proof, but since it relies on the lack of $Q_{j+1}$ in the \emph{open} interval $(Q_i,A_{i+1})$, it is interesting to consider the closed interval as well. There are two potential solutions to $B_i = Q_{j+1}$ in the closed interval $[Q_i,A_{i+1}]$, namely, $A_{i+1} = Q_{j+1}$ and $Q_i = Q_{j+1}$. For the former, we must use a value from $\Q$ in $\A$. For the latter, we have $j = i-1$, and from \eqref{B and C} and Proposition~\ref{T images}, we get that
    \begin{align*}
        B_i &= Q_i \\
        \hspace*{6em} T_{\sigma(i-1)} A_{\sigma(i-1)} &= Q_{i} \\
        A_{\sigma(i-1)} &= T_{\sigma(i-1)}^{-1} Q_{i} = T_{i-1} Q_{i} = P_{\sigma(i-1)},
    \end{align*}
    which is using a value from $\P$ in $\A$.
    Thus we see that combining the structure of $\Omega_\A$ given by short cycles with the existence of a dual forces $\A$ to use values from only $\P \cup \Q$.
\end{proof}

\section{Application to the entropy calculation}\label{sec entropy}

Following~\cite{KU12}, we use the representation of the geodesic flow on the compact surface $M=\G\backslash\D$ as the special flow over the arithmetic cross-section $C_\A$ and Abramov's formula to compute the measure-theoretic entropy of the two maps $F_\A$ and $f_\A$.

Let $(x,y,\psi)$, $z=x+iy\in\D,\,0\leq\psi<2\pi$, be the standard coordinate system on the unit tangent bundle $S\D$. The hyperbolic measures on $\D$ and $S\D$, corresponding to the metric (\ref{hypmetric}), are given by
\[ dA=\frac{4dxdy}{(1-{\abs z}^2)^2} \]
and $dAd\psi$, respectively. They are preserved by M\"obius transformations.

There is another coordinate system on $S\D$, introduced in~\cite{AF91} and used in~\cite{KU12}, which proved to be more convenient than $(x,y,\psi)$ in study of the geodesic flow, especially in the context of the cross-sections. Namely, to each $v\in S\D$ we assign the triple $(u,w,s)$. where $u$ and $w$ are unit circle variables and $s$ is real. The pair $(u,w)$ designate points of intersection of the geodesic determined by $v$ with the boundary $\partial\D=\{\,z : \abs z = 1 \,\}$; $u$ is the backward end and $w$ is the forward end of this geodesic. The real parameter $s$ is the hyperbolic length along the geodesic measured from its ``midpoint'' to the base point of $v$. As was pointed out in \cite{KU17}, it is a standard computation that the measure
\begin{equation}
    \label{AFhypmetric} d\nu=\displaystyle\frac{\abs{du}\,\abs{dw}}{\abs{u-w}^2}
\end{equation}
and the measure $dm=d\nu ds$ on $S\D$ are preserved by M\"obius transformations, in the first case applied to unit circle variables $u$ and $w$, and in the second case to the variables $(u,w,s)$. Therefore, $F_\A$ preserves the smooth probability measure
\begin{equation}
    \label{dnu} d\nu_\A=\frac{1}{K_\A} d\nu, \text{ where } K_\A=\int_{\Omega_\A} d\nu.
\end{equation}
This can be also be derived from the representation of the geodesic flow $\{\varphi^t\}$ on $\G\backslash S\D$ as the special flow over $\Omega_\A$ which parametrizes the arithmetic cross-section $C_\A$, as explained in Section~\ref{sec cross-section}, with $F_\A$ being the first return map to $\Omega_\A$ and the ceiling function $g_\A:\Omega_\A\to\R$ being the time of the first return to the cross-section $C_\A$ parametrized by $\Omega_\A$ (see Figure~\ref{fig first return}).

The circle map $f_\A$ is a factor of $F_\A$ (projecting on the $w$-coordinate), so one can obtain its smooth invariant measure $d\mu_\A$ by integrating $d\nu_\A$ over $ \Omega_\A$ with respect to the $u$-coordinate. 

We can immediately conclude that the systems $(F_\A, \nu_\A)$ and $(f_\A, \mu_\A)$ are ergodic from the fact that the geodesic flow $\{\varphi^t\}$ is ergodic with respect to $dm$.

The next result gives a formula for the measure-theoretic entropy of $(F_\A, \nu_\A)$.
Since $(F_\A, \nu_\A)$ is a natural extension of $(f_\A, \mu_\A)$, the measure-theoretic entropies of the two systems coincide, and we have the following result.
\begin{prop}
    $h_{\mu_{\!\A}}(f_\A)=h_{\nu_{\!\A}}(F_\A)=\frac{\pi^2(2g-2)}{K_\A}$.
\end{prop}

\begin{proof}
    Using the well-known fact that the entropy of the geodesic flow with respect to the normalized Liouville measure $d\tilde m=\frac{dm}{m(SM)}$ on $SM$ is equal to $1$ and Abramov's formula, we obtain
    \begin{equation}
        \label{Abramov} h_{\tilde m}(\{\varphi^t\})=1=\frac{h_{\nu_{\!\A}}(F_\A)}{\int_{\Omega_\A}g_\A d\nu_\A}.
    \end{equation}
    On the other hand, $d\tilde m$ can be represented by the Ambrose-Kakutani Theorem~\cite{AK42} as a smooth probability measure on the space $\Omega_\A^{\,g_{\!\A}}$ under the ceiling function $g_\A$, so we have
    \begin{equation}
        \label{AK42} d\tilde m=\frac{d\nu_\A ds}{\int_{\Omega_\A}g_\A d\nu_\A}=\frac{d\nu ds}{K_\A\int_{\Omega_\A}g_\A d\nu_\A}=\frac{dm}{m(SM)}.
    \end{equation}
    Combining (\ref{Abramov}) and (\ref{AK42}) and using the fact that the two measures $dm$ and $dAd\psi$, which are invariant under the geodesic flow on $S\D$, are related by
	\[ 4dm=dAd\psi, \]
	we obtain
    \[ h_{\nu_{\!\A}}(F_\A)=\int_{\Omega_{\A}}g_\A d\nu_\A=\frac{m(SM)}{K_\A}. \]
    But since by the Gauss-Bonnet formula $A(M)=\int_M dA=2\pi(2g-2)$, we have
    \begin{align*}
    	m(SM)
		&= \int_{SM} dm
		= \frac14\int_{SM} dAd\psi
		= \frac14(2\pi)\int_M dA
		\\&= \frac14(2\pi)\big(2\pi(2g-2)\big)
		= \pi^2(2g-2),
	\end{align*}
    and so $h_{\mu_{\!\A}}(f_\A)=h_{\nu_{\!\A}}(F_\A)=\frac{\pi^2(2g-2)}{K_\A}$.
\end{proof}

Thus, from the exact shape of the set $\Omega_\A$, the invariant measure and the entropy can be calculated precisely; one just needs to calculate the area $K_\A$ of the attractor. In~\cite[Proposition 7.1]{KU17} this invariant was calculated for the case when $\A$ satisfies the short cycle property.

\addtolength\textheight{1em}

\end{document}